\documentclass[onefignum,onetabnum]{siamart190516}
\usepackage{lipsum}
\usepackage{amsfonts,amsmath,mathtools}
\usepackage{graphicx}
\usepackage{epstopdf}
\usepackage{algorithmic}
\usepackage{enumitem}
\ifpdf
\DeclareGraphicsExtensions{.eps,.pdf,.png,.jpg}
\else
\DeclareGraphicsExtensions{.eps}
\fi

\newsiamremark{remark}{Remark}
\newsiamremark{hypothesis}{Hypothesis}
\crefname{hypothesis}{Hypothesis}{Hypotheses}
\newsiamthm{claim}{Claim}

\headers{HyQMOM for the 1-D Kinetic Equation}{R. O. Fox and F. Laurent}

\title{Hyperbolic Quadrature Method of Moments for the One-Dimensional Kinetic Equation\thanks{Submitted to the editors \today.}}

\author{Rodney O.~Fox\thanks{Department of Chemical and Biological Engineering, Iowa State University, 618 Bissell Road, Ames, IA 50011-1098, USA (\email{rofox@iastate.edu}).}
\and Fr\'ed\'erique Laurent\thanks{Laboratoire EM2C \& F\'ed\'eration de Math\'ematiques de CentraleSup\'elec, CNRS, CentraleSup\'{e}lec, Universit\'e Paris-Saclay, 3 rue Joliot-Curie 91192 Gif-sur-Yvette, France (\email{frederique.laurent@centralesupelec.fr}).}
}

\usepackage{amsopn}

\DeclarePairedDelimiter\floor{\lfloor}{\rfloor}

\usepackage{subcaption}
\usepackage{csquotes}
\usepackage{placeins}
\usepackage[algo2e,ruled,vlined]{algorithm2e}

\def\bfM{\mathbf{M}}

\def\bfF{\mathbf{F}}
\def\bfI{\mathbf{I}}
\def\bfJ{\mathbf{J}}
\def\bfK{\mathbf{K}}
\def\bfY{\mathbf{Y}}

\def\bbfC{\mathbf{C}}
\def\bbfS{\mathbf{S}}
\def\bfC{\widetilde{\mathbf{C}}}
\def\bfS{\widetilde{\mathbf{S}}}

\newcommand{\bP}{{\overline{P}}}
\newcommand{\bQ}{{\overline{Q}}}
\newcommand{\bR}{{\overline{R}}}
\newcommand{\balpha}{{\overline{\alpha}}}
\newcommand{\bbeta}{{\overline{\beta}}}
\newcommand{\ba}{{\bar a}}
\newcommand{\bb}{{\bar b}}

\def\RR{\mathbb{R}}

\newcommand{\bu}{{\bar u}}

\begin{document}

\maketitle

\begin{abstract}
A solution is proposed to a longstanding open problem in kinetic theory, namely, given any set of realizable velocity moments up to order $2n$, a closure for the moment of order $2n+1$ is constructed for which the moment system found from the free-transport term in the one-dimensional (1-D) kinetic equation is globally hyperbolic and in conservative form. In prior work, the hyperbolic quadrature method of moments (HyQMOM) was introduced to close this moment system up to fourth order ($n \le 2$). Here, HyQMOM is reformulated and extended to arbitrary even-order moments. The HyQMOM closure is defined based on the properties of the monic orthogonal polynomials $Q_n$ that are uniquely defined by the velocity moments up to order $2n-1$. Thus, HyQMOM is strictly a moment closure and does not rely on the reconstruction of a velocity distribution function with the same moments. On the boundary of moment space, $n$ double roots of the characteristic polynomial $P_{2n+1}$ are the roots of $Q_n$, while in the interior, $P_{2n+1}$ and $Q_n$ share $n$ roots. The remaining $n+1$ roots of $P_{2n+1}$ bound and separate the roots of $Q_n$. An efficient algorithm, based on the Chebyshev algorithm, for computing the moment of order $2n+1$ from the moments up to order $2n$ is developed. The analytical solution to a 1-D Riemann problem is used to demonstrate convergence of the HyQMOM closure with increasing $n$. 
\end{abstract}



\section{Introduction}\label{s:intro}

Moment closures in the context of kinetic theory have a long history. The 13-moment closure of Grad \cite{grad1949} and the moment-closure hierarchies of Levermore \cite{levermore96b} are two well-known examples. In general, finite-dimensional moment systems derived from a kinetic equation will have unclosed terms, usually involving higher-order moments that are not included in the moment vector. For example, the free-transport term in the Boltzmann equation generates an unclosed spatial flux for the moment of order $k$ that depends on the moment of order $k+1$. Broadly speaking, the closure of the latter can be accomplished (as proposed by Grad \cite{grad1949}) using a perturbative solution for the velocity distribution function (VDF) valid near the equilibrium distribution, or (as proposed by Levermore \cite{levermore96b}) using a non-perturbative reconstruction of the VDF such as entropy maximization \cite{mt2013}. 

Regardless of the method used to derive it, a well-posed moment closure must ensure that the predicted moments are realizable (i.e., they must be in the interior or on the boundary of the convex moment space \cite{h1944,dette97,Lasserre2010,Schmudgen2017}) and that the moment system derived from the kinetic equation is globally hyperbolic \cite{levermore96b,hly2020,flz2020,flz2020jcp}. Furthermore, the moment closure must be well defined for every set of realizable moments, in particular for every moment vector in the interior of moment space \cite{junk98,flz2020}, and preferably result in a moment system in conservative form.

One method to ensure realizability of the unknown moments is to reconstruct a non-negative VDF from the known (realizable) moments; however, this is insufficient to ensure that the resulting moment system will be globally hyperbolic. For example, a one-dimensional (1-D) kinetic equation can be modeled by a moment system up to an even-order ($2n$) moment, whose flux depends of the unknown moment of order $2n+1$. Since odd-order moments need only be finite to be realizable, any value can be chosen to close the flux (e.g., set the $2n+1$ central moment equal to zero). Moreover, it is always possible to reconstruct a non-negative VDF based on this choice (e.g., a weighted sum of Dirac delta functions \cite{w1974,gautschi04,Chebyshev1859}), but it is unlikely that such an arbitrary choice will result in a globally hyperbolic moment system \cite{hly2020}. 

Indeed, there is no need to require that a moment closure corresponds to a \emph{particular} form of the VDF (except on the boundary of moment space where the VDF is unique \cite{Schmudgen2017}). In other words, the ability to reconstruct the VDF from its moments is not a necessary condition for global hyperbolicity. Conversely, poor choices for the moment closure can generate unrealizable moments during the time advancement of the moment system, even if the moment system is hyperbolic and in conservative form \cite{vikas2011}. Thus, care must be taken to guarantee that moments on the boundary of moment space do not leave it due to the closure for the flux.

For classical applications of the Boltzmann equation \cite{t2016}, it is often acceptable to employ a moment closure that is realizable and conditionally hyperbolic in a subset of moment space \cite{struch05}  (e.g., near the moments of the equilibrium VDF). A recent example of such a closure applied to rarefied gas dynamics can be found in \cite{ks2020}. For the same application, other authors have `regularized' or `modified' the moment system by adding non-conservative terms (see, for example, \cite{cfl2013,cfl2015,kf2020,flz2020,fm2019}) to achieve global hyperbolicity. In doing so, the direct connection between the closure and the reconstructed VDF is lost. 

As an example, for moment vectors on the boundary of moment space where the VDF is unique \cite{Schmudgen2017}, it remains to be shown whether or not the eigenvalues of the modified moment system found using the procedure in \cite{kf2020} correspond to the exact values found in \cite{kah11_cms,hly2020}. The latter are equal to the roots of the $n$th-order orthogonal polynomial $Q_n$, which can be found from the moments using the Chebyshev algorithm \cite{w1974,gautschi04,Chebyshev1859}. In the context of dilute sprays and particulate flows without collisions \cite{fox08b,fox09b,dechaisemartin2009cras,kah11_ctr,mf2013,cflmv2017,flv2018,fm2019,hly2020}, moment systems very close to or on the boundary of moment space are regularly encountered. Thus, a robust moment closure must be able to describe accurately the evolution of moment sets near the boundary. 

Hereinafter, we consider the velocity moments for a 1-D kinetic equation with only the free-transport term, which suffices for testing whether the proposed moment closure is globally hyperbolic. Further, the moment system consists of the $2n+1$ moments up to order $2n$, so that the spatial flux requires a closure for the moment of order $2n+1$.  In prior work \cite{flv2018}, the hyperbolic quadrature method of moments (HyQMOM) was developed for $n \le 2$, and shown to yield a realizable and globally hyperbolic moment system in conservative form. In that work, the motivation for closing the spatial flux was to fix the middle root of the polynomial $Q_3$ at the mean velocity $\bu$, which leads to a unique choice for the fifth-order velocity moment and corresponds to a VDF as a weighted sum of Dirac delta functions with one abscissa at $\bu$. With this choice, it is possible to compute the five eigenvalues of the moment system analytically and prove global hyperbolicity, as well as realizability. 

In the present work, we extend HyQMOM to $2 \le n$ by choosing a closure for the ($2n+1$)-order moment in conservative form that results in global hyperbolicity. In the process, we again make use of the polynomials $Q_n$ (that are determined from the $2n$ moments up to order $2n-1$) to define a monic polynomial $R_{n+1}$ of order $n+1$ whose roots separate and bound those of $Q_n$. The Chebyshev algorithm \cite{w1974,gautschi04,Chebyshev1859} is used to find $Q_n$ from a set of realizable moments.  We then demonstrate that $R_{n+1}$ is defined such that the characteristic polynomial of the 1-D moment system is $P_{2n+1} = Q_n R_{n+1}$. For $n=2$, this procedure results in the middle root of $Q_3$ not being equal to $\bu$, and 
leads to different locations for the eigenvalues for moment vectors in the interior of moment space than in \cite{flv2018}. Nonetheless, the basic idea for closing the ($2n+1$)-order moment using the properties of $Q_n$ remains unchanged, and thus we will continue to refer to the proposed extension as HyQMOM.

In the context of moment closures for the kinetic equation, HyQMOM is a \emph{pure} moment closure in the sense that it does not rely on the reconstruction of a VDF to define the closure.  In theory, an explicit formulae can be written for the HyQMOM closure that depends on the recurrence formula for the orthogonal polynomials $Q_n$. However, for large $n$, this expression is quite complex and must be computed numerically.  For this purpose, in \cref{sec:mc}, we develop an efficient algorithm for computing the HyQMOM closure based on the Chebyshev algorithm. Before describing the proposed closure, in \cref{s::mom} we provide background information on the VDF and its moments. In \cref{s:quad-1D}, we introduce the 1-D kinetic equation and prove a theorem that relates the characteristic polynomial of the moment system $P_{2n+1}$ to the functional form of the moment closure for the standardized moment $S_{2n+1}$.  Then, in \cref{sec:mc}, we define the HyQMOM closure for arbitrary $n$ by making use of the polynomials $Q_n$ and their recurrence coefficients $a_n$ and $b_n$.  

The remainder of the paper is then devoted to exploring the properties of the proposed closure.  
In \cref{sec:example}, examples of the HyQMOM closure for $n \le 5$ are provided to illustrate the behavior of the roots of the characteristic polynomial for selected moment sets. Then, in \cref{sec:tc}, a Riemann problem with free transport is solved numerically using the HyQMOM closure and compared to the analytical solution for the moments. For increasing $n$, we demonstrate numerically that the HyQMOM closure converges uniformly towards the analytical solution.  Finally, in \cref{sec:con}, conclusions are drawn and possible future extensions of the HyQMOM closure to multidimensional and bounded domains are briefly discussed.

\section{Velocity distribution function and its moments}\label{s::mom}

Consider the 1-D VDF $f(t,x,u)$ defined for $(t,x,u) \in \mathbb{R}^+ \times \mathbb{R} \times \mathbb{R}$. The VDF is non-negative and its moments with respect to the velocity $u$ are finite. For a symmetric VDF about the mean velocity $\bu$, $f(t,x,u-\bu) = f(t,x,\bu-u)$.

\subsection{Moments}

The moments of $f$ are defined by 
\begin{equation}
M_k (t,x) :=  \int_{\mathbb{R}} f(t,x,u) u^k \, \mathrm{d}u \quad \text{for $k \in \{0, 1, \dots, 2n+1 \}$}
\end{equation}
and $n \in \mathbb{N}$. Let us also denote the moment vector by $\bfM_{2n} := (M_0,M_1,\dots,M_{2n})^t$. The moment $M_{2n+1}$ is not included in the moment vector, but we wish to specify it as depending on $\bfM_{2n}$ subject to some constraints. This is the principal challenge faced in moment closures for the kinetic equation.

\subsection{Central moments}
\label{ssec:cmom}

For $M_0>0$, the central moments are defined by
\begin{equation}\label{eq:M2C}
	C_k (t,x) := \frac{1}{M_0} \int_{\mathbb{R}} f(t,x,u) (u - \bu)^k \, \mathrm{d}u \quad \text{for $k \in \{0, 1, \dots, 2n+1\}$}
\end{equation}
where $\bu (t,x) = M_1/M_0$ is the mean velocity.  
By definition, $C_0=1$ and $C_1=0$.  
The next central moment $C_2 \ge 0$ is the velocity variance. 
Let us remark that the vector $\bbfC_k := (C_0,\dots,C_k)^t=(1,0,C_2,\dots,C_k)^t$ is the $k^{th}$-order moment vector corresponding to $v\mapsto \dfrac{1}{M_0}f(\bu+v)$.

For any $k \ge 2$, the central moment $C_k$ depends uniquely on the moment set ${\bfM_k:= (M_0,M_1,\dots,M_{k})}$ through the relation
\begin{equation}\label{eq:central}
	C_k  = \sum_{i=0}^k \binom{k}{i} \left(-\frac{M_1}{M_0}\right)^{k-i} \frac{M_i}{M_0}.
\end{equation}
And inversely, for $k\ge 2$, $M_k$ depends uniquely on the central moment vector ${\bfC_k}=(M_0,\bu,C_2,\dots,C_k)^t$ through the relation
\begin{equation}\label{eq:C2M}
	M_k  = M_0\left[  \sum_{i=2}^k \binom{k}{i} \bu^{k-i} C_i + \bu^k \right].
\end{equation}
In particular, the unclosed moment $M_{2n+1}$ can be written in terms of the components of $\bfC_{2n}$ as
\begin{equation}\label{eq:C2p}
M_{2n+1}  = M_0\left[  \sum_{i=2}^{2n} \binom{2n+1}{i} \bu^{2n+1-i} C_i + \bu^{2n+1} + C_{2n+1} \right] 
\end{equation}
where $C_{2n+1}$ must be specified as an algebraic function of ${\bfC_{2n}}$. Thus, in the context of moment closures, specifying $C_{2n+1}$ is equivalent to specifying $M_{2n+1}$.

\subsection{Standardized moments}
\label{ssec:smom}

For $C_2 > 0$, the standardized moments are defined by
\begin{equation}\label{eq:scaledcentral}
S_k := \frac{C_k}{C_2^{k/2}} \quad 
\text{for $k \in \{0,1, \dots, 2n+1 \}$}.
\end{equation}
By definition, $(S_0,S_1,S_2)=(1,0,1)$ and 
$\bbfS_k := (S_0,\dots,S_k)^t=(1,0,1,S_3,\dots,S_k)^t$ is the $k^{th}$-order moment vector corresponding to $v\mapsto \dfrac{\sqrt{C_2}}{M_0}f(\bu+\sqrt{C_2}v)$ .
When $C_2=0$, all higher-order central moments are null. Unless stated otherwise, hereinafter we assume that $0 < C_2$ and define the standardized moment vector $\bfS_k=(M_0,\bu,C_2,S_3,\dots,S_k)^t$, which has a one-to-one relation with $\bfM_k$ in the interior of moment space. 
In the context of moment closures, we must specify $S_{2n+1}$ for $2 \le n$ as an algebraic function of $\bfS_{2n}$ subject to constraints.

\subsection{Realizability}

In this work, we will not attempt to reconstruct the VDF from its moments. Instead, we approximate the moment of order $2n+1$ given lower-order moments, i.e., we seek a particular solution to the truncated Hamburger moment problem \cite{h1944}. Nonetheless, given the data $(M_0,\bu,C_2,S_3, \dots S_{2n}, S_{2n+1})$ with $M_0>0$, $C_2>0$, a reconstructed VDF consisting of a sum of weighted Dirac delta functions can be found when the moment set is realizable. The latter is verified using the Hankel determinants \cite{Schmudgen2017}:
\begin{equation}\label{eq:hankel}
H_{2k} =
\begin{vmatrix}
1      & 0   & 1   & S_3 & \cdots  & S_{k} \\
0      & 1   & S_3 & S_4 &        &  \\
1      & S_3 & S_4 & S_5 &        &  \\
S_3    & S_4 & S_5 & S_6 &        &       \\
\vdots &     &     &     & \ddots &       \\
S_{k}  &     &     &     &        & S_{2k}
\end{vmatrix} ,
\qquad k\in \{2,\dots,n\}.
\end{equation}
If $H_{2k}>0$ for $k\in \{2,\dots,n\}$, then the moment set is realizable and lies in the interior of moment space. It is then said to be strictly realizable.
If $H_{2k} < 0$ for any $k \in \{2,\dots,n\}$ then the moment set is unrealizable. 
Otherwise, the moment set is realizable if and only if there exists $k \in \{ 2,\dots,n \}$ such that $H_4>0,\dots,H_{2k-2}>0$, $H_{2k}=\dots=H_{2n}=0$ and 
the vector $(S_{n+1},\dots,S_{2n+1})$ is in the span of $(S_j,S_{j+1},\dots,S_{j+n})$, $j=0,\dots,n$.
In this last case, the moment set then lies on the boundary of moment space and the VDF is exactly a weighted sum of $k$ Dirac delta functions.

Note that $0 = H_{2n}$ defines the lower bound for $S_{2n}$ (i.e., the even-order moments). In contrast, for the Hamburger moment problem, if $H_{2k}>0$ for $k\in \{2,\dots,n\}$, the odd-order moment $S_{2n+1}$ can take on any finite real value.
In the context of moment closures, we require that the algebraic function defining $S_{2n+1}$ be valid for any vector of realizable moments $\bfM_{2n}$.

\subsection{Linear functional associated with a moment vector} \label{ssec:lin}
For a moment vector $\bfM_N$, one can define the linear functional $\langle . \rangle_{\bfM_N}$ on the space $\mathbb{R}[X]_{N}$ of the real polynomial function of degree smaller than $N$ by 
\begin{equation}\label{eq:defangle}
\langle X^k \rangle_{\bfM_N}=M_k, \quad  \text{for $k \in \{0, 1, \dots, N\}$}.
\end{equation}
Let us remark that if $\bfM_N$ is associated with a VDF $f$, then 
\begin{equation}
\forall P\in\mathbb{R}[X]_{N} \qquad  \langle P \rangle_{\bfM_N} =\int_{\mathbb{R}} P(u)f(u) \mathrm{d} u.
\end{equation}
Let us consider the linear functionals $\langle . \rangle_{\bbfS_N}$ associated with the standardized moments defined in \cref{ssec:smom}.
It is linked to the one associated with $\bfM_N$ by
\begin{equation}\label{eq:link_func}
\forall P\in\mathbb{R}[X]_{N} \qquad  \langle P(X) \rangle_{\bfM_N} = M_0\left\langle P\left(\bu+\sqrt{C_2}X\right) \right\rangle_{\bbfS_N}.
\end{equation}
In the following, the notation $\langle . \rangle$ is used for simplicity and corresponds to $\langle . \rangle_{\bbfS_{2n}}$.
Indeed, it is easier to work with $\bfS_{2n}$ than directly with $\bfM_{2n}$, and it is equivalent as soon as $M_0>0$ and $C_2>0$.
Moreover, as soon as $\bfM_{2n}$ (or equivalently $\bbfS_{2n}$) is strictly realizable, the application $(p,q)\mapsto \langle pq \rangle$ defines a scalar product on $\mathbb{R}[X]_{n}$.

\section{Kinetic equation and its moment system}\label{s:quad-1D}

The 1-D kinetic equation for the VDF including only free transport is
\begin{equation} \label{equ:cinetique}
\partial_t f + u \partial_x f  = 0, 
\end{equation}
with initial condition $f(0,x,u) = f_0(x,u)$. The exact solution is given by $f(t,x,u) = f(0,x-ut,u) = f_0(x-ut,u)$. In this work, we seek to approximate the moments of the VDF found from \eqref{equ:cinetique} by formulating a finite-dimensional moment system.

\subsection{Moment system}

The unclosed governing equations for the moment vector are 
\begin{equation} \label{modele_3pic}
\begin{aligned}
\partial_t M_0 + \partial_x M_1 &= 0, \\
\partial_t M_1 + \partial_x M_2 &= 0, \\
\vdots & \\
\partial_t M_{N} + \partial_x M_{N+1} &= 0  ;
\end{aligned}
\end{equation}
or, in vector form,
\begin{equation}\label{eq:closure5}
\partial_t \bfM_{N} + \partial_x \bfF (\bfM_{N} ) = \boldsymbol{0}
\end{equation}
where the unclosed flux vector is $\bfF (\bfM_{N} ) = (M_1, M_2, \dots, M_{N}, M_{N+1} )^t$. A viable moment closure provides an algebraic function $M_{N+1}(\bfM_{N})$ that is well defined for any realizable moment vector $\bfM_{N}$. It then must be demonstrated that such a closure is globally hyperbolic. In this work, the characteristic polynomial derived from \eqref{eq:closure5} will be used for this purpose. The system \eqref{eq:closure5} with the closed flux vector will then have a conservative hyperbolic form.

\subsection{Characteristic polynomial}

The Jacobian matrix of system \eqref{eq:closure5} is given by 
\begin{equation}\label{Jacmatrix1}
	\frac{D \bfF}{D \bfM} =
	\begin{pmatrix}
		0      & 1     & 0      & \dots & 0   \\
		0      & 0     & 1      &       &  0 \\
		\vdots &       & \ddots & \ddots &  \vdots\\
		0      & 0     &  0     &    0    & 1 \\
		\frac{\partial M_{N+1}}{\partial M_0} & 
		\frac{\partial M_{N+1}}{\partial M_1} & 
		\frac{\partial M_{N+1}}{\partial M_2} &\dots &
		\frac{\partial M_{N+1}}{\partial M_{N}} 
	\end{pmatrix} .
\end{equation} 
Equivalently, as soon as $M_0>0$, the variable set $\bfC_{N}=(M_0,\bu,C_2,\dots,C_{N})^t$ can be used, as well as $\bfS_{N}=(M_0,\bu,C_2,S_3,\dots,S_{N})^t$ if in addition $M_2M_0>M_1^2$, where the central moments $C_k$ and the standardized moments $S_k$ are defined by \eqref{eq:central} and \eqref{eq:scaledcentral}, respectively.
The closure is then given by the algebraic function $C_{N+1}(\bfC_{N})$ or $S_{N+1}(\bfS_{N})$, and \eqref{eq:C2p}.

In terms of the central moments, as soon as the variables are regular enough, the system can be rewritten:
\begin{equation}
	\partial_t \bfC_{N} + \bfJ \partial_x \bfC_{N} = \boldsymbol{0}
\end{equation}
with 
\begin{equation}\label{Jacmatrix}
	\bfJ := 
	\frac{D \bfM}{D \bfC}^{-1}
	\frac{D \bfF}{D \bfM}
	\frac{D \bfM}{D \bfC} .
\end{equation}
Of course, the characteristic polynomial of $\frac{D \bfF}{D \bfM}$ is the same as the one of $\bfJ$ and is denoted 
$\bP_{N+1}(X):=|\bfJ-X\bfI |$.
The following proposition can then be shown.
\begin{proposition}\label{th:dependence}
	Let $\bfM_{N}=(M_0,M_1,M_2,\dots,M_{N})^t$ be a realizable moment vector such that $M_0>0$ and $M_2M_0>M_1^2$, and let 
	$\bfS_{N} =(M_0,\bu,C_2,S_3,\dots,S_{N})^t$ be the corresponding standardized moment vector.
	Let us assume that the function $S_{N+1}$ does not depend on $(M_0,\bu,C_2)$, i.e., $S_{N+1}(S_3,\dots,S_{N})$.
	Then, the following polynomial 
	\begin{multline}\label{eq:polychar}
		P_{N+1}(X) := \bP_{N+1}\left(\bu + C_2^{1/2} X\right) C_2^{-(N+1)/2} \\
		=\left| \bfJ - \left( \bu + C_2^{1/2} X \right) \bfI \right| C_2^{-(N+1)/2} ,
	\end{multline}
	where $\bfJ$ is defined by \eqref{Jacmatrix},
	only depends on $(S_3,\dots,S_{N})$.
\end{proposition}
\begin{proof}
	The equations for $M_0$ and $\bu$ are 
	\begin{equation}\label{eq:M0u}
		\begin{gathered}
			\partial_t M_0 + \bu\partial_x M_0 +M_0\partial_x \bu =0,\\
			\partial_t \bu + \frac{C_2}{M_0} \partial_x M_0  + \bu \partial_x \bu +\partial_x C_2=0.
		\end{gathered}
	\end{equation}
	Since the moment vector $\bfM_{N+1}(t,x)$ is realizable for each $(t,x)$, it can be associated with a VDF $f(t,x,u)$.
	Setting $u=v+\bu(t,x)$, the function $\tilde f(t,x,v) :=f(t,x,v+\bu(t,x))$ is then such that 
	\begin{equation}\label{eq:changvar}
		\partial_t f +u\partial_x f=\partial_t \tilde f +v \partial_x  \tilde f + \bu\partial_x \tilde f-(\partial_t \bu+\bu \partial_x \bu+v \partial_x \bu)\partial_v \tilde f.
	\end{equation}
	The central moments are such that $M_0C_k=\int_{\mathbb{R}} v \tilde f(t,x,v) \mathrm{d} v$ for $k=2,\dots,N$.\\
	Since $\int_{\mathbb{R}} u^k (\partial_t f +u\partial_x f)(t,x,u) \mathrm{d}u=0$, for $k=2,\dots,N$, then, thanks to \eqref{eq:changvar}:
	\begin{multline}
		\partial_t (M_0C_k) 
		+ \partial_x (M_0C_{k+1}) +\bu\partial_x(M_0C_k) 
		+(\partial_t \bu+\bu \partial_x \bu)kM_0C_{k-1} \\
		+ (\partial_x\bu)(k+1)M_0C_k=0.
	\end{multline}
	This equation can be rewritten, using \eqref{eq:M0u}:
	\begin{equation}
		\partial_t C_k 
		+ \frac{C_{k+1}-kC_2C_{k-1}}{M_0}\partial_x M_0
		+ kC_k\partial_x \bu
		- kC_{k-1}\partial_x C_2
		+ \bu\partial_x C_k + \partial_x C_{k+1}=0.
	\end{equation}
	This result, along with \eqref{eq:M0u}, allows to write the matrix $\bfJ$.
	Then, $\left| \bfJ - \left( \bu + C_2^{1/2} X \right) \bfI \right|$ can be written as
	\begin{multline*}
		\left| \quad
	\begin{matrix}
		-X\sqrt{C_2} & M_0 & 0 & \\
		\frac{C_2}{M_0} & -X\sqrt{C_2}  & 1 &\\
		\frac{C_3}{M_0} & 2C_2 &  -X\sqrt{C_2}  &\\
		\frac{C_4-3C_2C_2}{M_0} & 3C_3 &  -3C_2  &\\
		\vdots & \vdots & \vdots &\\
		\frac{C_{N-1}-(N-2)C_2C_{N-3}}{M_0} & (N-2)C_{N-2} &  -(N-2)C_{N-3}  &\\
		\frac{C_{N}-(N-1)C_2C_{N-2}}{M_0} & (N-1)C_{N-1} &  -(N-1)C_{N-2}  &\\
		\frac{C_{N+1}-NC_2C_{N-1}}{M_0} & NC_{N} &  -NC_{N-1} +\frac{\partial C_{N+1}}{\partial C_2}  
		&\\
	\end{matrix} \right. \\
	\left.
	\begin{matrix}
	0   & \cdots & 0 & 0 \\
	0   & \cdots & 0 & 0\\
	1   & \cdots & 0 & 0\\
	-X \sqrt{C_2}  & \ddots & 0 & 0\\
	    & \ddots &   & \vdots \\
	 0  &    & 1 & 0\\
	 0  &       & -X\sqrt{C_2}& 1\\
	 \frac{\partial C_{N+1}}{\partial C_3}   & \dots & \frac{\partial C_{N+1}}{\partial C_{N-1}} & \frac{\partial C_{N+1}}{\partial C_{N}}-X\sqrt{C_2}\\
	\end{matrix} \quad \right| .
	\end{multline*}
	Let us remark that, thanks to the assumption on $S_{N+1}$, $C_{N+1}(C_2,\dots,C_{N})$.
	Factoring $C_2^{1/2}/M_0$ to the first row, $C_2^{i/2}$ to each row $i>1$ and $M_0$ to the first column and  $1/C_2^{(j-1)/2}$ to each column $j>1$, we obtain
	\begin{multline*}
	C_2^{(N+1)/2} \left| \quad
	\begin{matrix}
		-X & 1 & 0  \\
		1 & -X  & 1 \\
		S_3 & 2 &  -X  \\
		S_4-3S_2 & 3S_3 &  -3  \\
		\vdots & \vdots & \vdots  \\
		S_{N-1}-(N\!\!-\!\!2)S_{N-3} & (N\!\!-\!\!2)S_{N-2} &  -(N\!\!-\!\!2)S_{N-3}  \\
		S_{N}-(N\!\!-\!\!1)S_{N-2} & (N\!\!-\!\!1)S_{N-1} &  -(N\!\!-\!\!1)S_{N-2}  \\
		S_{N+1}-NS_{N-1} & NS_{N} &  -NS_{N-1} +\frac{N+1}{2}S_{N+1} \\
	\end{matrix} \right. \\
	\left.
	\begin{matrix}
		 0 & \cdots & 0 & 0 \\
		 0 & \cdots & 0 & 0\\
		 1 & \cdots & 0 & 0\\
		 -X& \ddots & 0 & 0\\
		   & \ddots &   & \vdots  \\
		 0 &        & 1 & 0\\
		 0 &        & -X& 1\\
		 \frac{\partial S_{N+1}}{\partial S_{3}} &
		 \dots & \frac{\partial S_{N+1}}{\partial S_{N-1}} & \frac{\partial S_{N+1}}{\partial S_{N}}-X\\
	\end{matrix} \quad \right| .
	\end{multline*}
	This result shows that $P_{N+1}(X)$ only depends on $(S_3,\dots,S_{N})$, thus concluding the proof.
\end{proof} 

The eigenvalues of \eqref{eq:closure5} are then written $\lambda_k = \bu + C_2^{1/2} \mu_k$, where $\mu_k$ is a root of $P_{N+1}$ that only depends on $(S_3,\dots,S_{N})$.
It is then easy to calculate the coefficients of the characteristic polynomial:
\begin{theorem}\label{theorem1}
	With the same assumptions as in \cref{th:dependence}, the scaled characteristic polynomial defined by \eqref{eq:polychar} has the form
	\begin{equation}\label{eq:PX}
		P_{N+1}(X) = \sum_{m=0}^{N+1} c_{m} X^{m} 
	\end{equation}
	with coefficients defined by
	\begin{equation}\label{charPcoeff}
		\begin{gathered}
			c_{N+1} = 1, \ c_{N} = - \frac{\partial S_{N+1}}{\partial S_{N}}, \
			c_{N-1} = - \frac{\partial S_{N+1}}{\partial S_{N-1}}, \ \cdots, \
			c_{3} = - \frac{\partial S_{N+1}}{\partial S_{3}}, \\
			c_2 = - \frac{1}{2} \sum^{N+1}_{m=3} m S_{m} c_{m} , \quad
			c_1 = - \sum_{m=3}^{N+1} m S_{m-1} c_{m}  , \\ 
			c_0 = \frac{1}{2} \sum^{N+1}_{m=3} (m-2) S_{m} c_{m}  = - \, c_2 
			 - \sum^{N+1}_{m=3} S_{m} c_{m}  .
		\end{gathered}
	\end{equation}
\end{theorem}
\begin{proof}
	We denoted $\bbfS_{N}:=(1,0,1,S_3,\dots,S_{N})^t$.
	From \cref{th:dependence}, it is known that $P_{N+1}$ does not depend on $(M_0,\bu,C_2)$, so that we can choose $(M_0,\bu,C_2)=(1,0,1)$ in such a way that $\bfS_{N}=\bbfS_{N}$. 
	This implies that the corresponding moment vector is $\bfM_{N}=\bbfS_{N}$, as well as for the central moment vector: $\bfC_{N}=\bbfS_{N}$.
	Then $P_{N+1}$ is equal to the characteristic polynomial of $\frac{D \bfF}{D \bfM}$ for $\bfM_{N}=\bbfS_{N}$.
	Thus, from \eqref{Jacmatrix1}, the coefficients of this polynomial can be written:
	\begin{multline*}
	c_{N+1} = 1, \quad 
	(c_0,c_1,\dots,c_{N}) = -\left.\frac{D M_{N+1}}{D\bfM}\right|_{\bfM_{N}=\bbfS_{N}} \\
	=  -\left.\frac{D M_{N+1}}{D\bfS}\right|_{\bfM_{N}=\bbfS_{N}} \left(\left.\frac{D M}{D\bfS}\right|_{\bfM_{N}=\bbfS_{N}} \right)^{-1} .
	\end{multline*}
	Moreover, from \eqref{eq:C2M}, we can write
	$$
	\left.\frac{D M}{D\bfS}\right|_{\bfM_{N}=\bbfS_{N}} = 
	\begin{pmatrix}
		1 & 0 & 0 & 0 & \cdots & 0\\
		0 & 1 & 0 & 0 & \cdots & 0\\
		1 & 0 & 1 & 0 & \cdots & 0\\
		S_3 & 3S_2 & \frac{3}{2}S_3 & 1 & \cdots & 0\\
		\vdots & \vdots & \vdots &  & \ddots & \\
		S_{N} & NS_{N-1} & \frac{N}{2}S_{N} & 0 & \cdots & 1\\
	\end{pmatrix} ,
	$$
	and then
	$$
	\left(\left.\frac{D M}{D\bfS}\right|_{\bfM_{N}=\bbfS_{N}} \right)^{-1} = 
	\begin{pmatrix}
		1 & 0 & 0 & 0 & \cdots & 0\\
		0 & 1 & 0 & 0 & \cdots & 0\\
		-1 & 0 & 1 & 0 & \cdots & 0\\
		(\frac{3}{2}-1)S_3 & -3S_2 & -\frac{3}{2}S_3 & 1 & \cdots & 0\\
		\vdots & \vdots & \vdots &  & \ddots & \\
		(\frac{N}{2}-1)S_{N} & -NS_{N-1} & -\frac{N}{2}S_{N} & 0 & \cdots & 1
	\end{pmatrix} .
	$$
	Furthermore, from \eqref{eq:C2p}, we can write
	$$
	\left.\frac{D M_{N+1}}{D\bfS}\right|_{\bfM_{N}=\bbfS_{N}} 
	= \left(S_{N+1},(N+1)S_{N},\frac{N+1}{2}S_{N+1},\frac{\partial S_{N+1}}{\partial S_3},\dots,\frac{\partial S_{N+1}}{\partial S_{N}}\right).
	$$
	Multiplying this row by the previous matrix concludes the proof.
\end{proof}

\subsection{Properties of the characteristic polynomial}

From the definition in \cref{ssec:lin} of the linear functional $\langle . \rangle$ associated with the moment vector $\bbfS_{N}$, one has
\begin{equation}\label{eq:pX2}
\langle P_{N+1} \rangle = 
c_{N+1} S_{N+1} + c_{N} S_{N} + \dots  + c_3 S_3 + c_2 + c_0 .
\end{equation}
If we define the polynomial $P_{N+1}'(X)$  by
\begin{equation}\label{Appell2}
P_{N+1}'(X) 
= \sum_{m=1}^{N+1} m c_{m} X^{m-1}  ,
\end{equation}
then the properties \eqref{charPcoeff} of the coefficients of the scaled characteristic polynomial provide directly the following three constraints:
\begin{corollary}
The scaled characteristic polynomial $P_{N+1}$ defined by \eqref{eq:polychar} is such that
\begin{equation}\label{3const}
	\langle P_{N+1} \rangle = 0, \quad 
	\langle P_{N+1}' \rangle = 0, \quad
	\langle X P_{N+1}' \rangle = 0 .
\end{equation}
\end{corollary}
As shown in \cref{HyQMOM}, these constraints are used to define the HyQMOM closure.

\section{Quadrature-based moment closure for $S_{2n+1}$}\label{sec:mc}

In general, QBMM provide a closure for higher-order moments in terms of the known lower-order moments. For example, $M_{2n}$ given $\bfM_{2n-1}$, or $M_{2n+1}$ given $\bfM_{2n}$. Without loss of generality, we describe QBMM using the standardized moments. Before defining the moment closure for $S_{2n+1}$ used in HyQMOM, we first review the quadrature method of moments (QMOM), which provides a closure for $S_{2n}$ in such a way that $\bbfS_{2n}$ is on the boundary of moment space.

\subsection{QMOM closure for $S_{2n}$ and orthogonal polynomials}

QMOM considers moments up to $S_{2n-1}$ with the closure for $S_{2n}$ found from $H_{2n}=0$, i.e., on the boundary of moment space. 
In this case, the unique VDF has the form of a sum of weighted Dirac delta functions located at the roots of the polynomial $Q_n$ defined just below.
Indeed, to compute the quadrature points, it is interesting to introduce the family of monic orthogonal polynomials $Q_n$, $\text{deg}(Q_n)=n$, for the scalar product defined in \cref{ssec:lin}, i.e., $\langle Q_m Q_n \rangle = \langle Q_n^2 \rangle \delta_{mn}$.
Moreover, $\langle Q_n^2 \rangle = H_{2n}/H_{2n-2}$.
This family satisfies the following recurrence relation 
\begin{equation}\label{Qn+1}
	Q_{n+1}(X) = (X - a_n) Q_n(X) - b_n Q_{n-1}(X)
\end{equation}
with $Q_{-1} = 0$ and $Q_{0} = 1$.
The recurrence coefficients $a_n$ and $b_n$ can be found from the standardized moments using the Chebyshev algorithm \cite{w1974,gautschi04,Chebyshev1859}, which is given in \cref{CBA}.
They are related to the orthogonal polynomials by
\begin{equation}\label{Qcoeff}
	a_n = \frac{\langle X Q_n^2 \rangle}{\langle Q_n^2 \rangle}, \quad
	b_n = \frac{\langle Q_n^2 \rangle}{\langle Q_{n-1}^2 \rangle} 
	    = \frac{H_{2n} H_{2n-4}}{H_{2n-2}^2}.
\end{equation}
The first few are
	$a_0 = 0 $, 
	$a_1 = S_3 $, 
	$a_2 = \frac{S_5  -  S_3 (2 + S_3^2 + 2 H_4)}{H_4} $, 
	$b_0 = 1 $,
	$b_1 = 1 $, 
	$b_2 = H_4 $, and
	$b_3 = {H_6}/{H_4^2} $.
For QMOM with a given $n$, starting from the standardized moments up to $S_{2n-1}$, the Chebyshev algorithm computes the recurrence coefficients up to $a_{n-1}$ and $b_{n-1}$ to define $Q_n$. Note that $a_n$ depends on standardized moments up to $S_{2n+1}$, with the highest-order moment having a linear dependence; and $b_n$ depends on standardized moments up to $S_{2n}$. The $b_n$ are positive except at the boundary of moment space where they can be zero.

\begin{remark}
	For Gaussian moments, $Q_n$ is the monic Hermite polynomial $He_n$.
\end{remark} 

\begin{remark}\label{rem:linkab}
	The relation (\ref{eq:link_func}) between $\langle.\rangle$ and $\langle.\rangle_{\bfM_{2n}}$ induces that the orthogonal monic polynomials $\bQ_k$ related to the scalar product $(p,q)\mapsto \langle pq\rangle_{\bfM_{2n}}$ on $\RR[X]_{2n}$ are such that 
	\begin{equation}\label{eq:linkQ}
	\bQ_k(X)=C_2^{k/2}Q_k\left(\frac{X-\bu}{\sqrt{C_2}}\right) ,
	\end{equation}
	and the corresponding coefficients $\ba_n$ and $\bb_n$ are
	\begin{align}\label{eq:linka}
	\ba_n &= \frac{\langle X \bQ_n^2 \rangle_{\bfM_{2n}}}{\langle \bQ_n^2 \rangle_{\bfM_{2n}}} 
	= \frac{\langle \bu +\sqrt{C_2}X Q_n^2 \rangle}{\langle Q_n^2 \rangle} 
	= \bu +\sqrt{C_2}a_k,\\
	\label{eq:linkb}
	\bb_n &= \frac{\langle \bQ_n^2 \rangle_{\bfM_{2n}}}{\langle \bQ_{n-1}^2 \rangle_{\bfM_{2n}}} 
	= C_2\frac{\langle Q_n^2 \rangle}{\langle Q_{n-1}^2 \rangle} 
	= C_2 b_n.
        \end{align}
\end{remark} 

In general, if the standardized moments correspond to a strictly realizable moment set, then there exists a one-to-one relationship between the moment vector $\bbfS_{2n-1}$
and the recurrence coefficients $(a_1,\dots, a_{n-1}, b_2,\dots,b_{n-1})$. 
Thus, the quadrature-based moment closure can be expressed equivalently in terms of the standardized moments or the recurrence coefficients.  
We will use this fact when defining the HyQMOM closure in \cref{HyQMOM}.
Moreover, for all $n=1,2,\dots$; $S_{2n}$ and $S_{2n+1}$ expressed in terms of $\bfY = (a_1, b_2, a_2, \dots, a_{n-1}, b_n, a_n)$ are multivariate polynomials (see \cref{CBA}), as are the components of the vectors of partial derivatives $\dfrac{D S_{2n}}{D \bfY}$ and $\dfrac{D S_{2n+1}}{D \bfY}$.

Finally, the QMOM closure for $S_{2n}$ corresponds to setting $b_n=0$ to find $S_{2n}$.
The form of the characteristic polynomial $P_{2n}$ for the corresponding system were given in \cite{kah11_cms,hly2020}.
\begin{theorem}
The QMOM closure $b_n=0$ induces the following characteristic polynomial $P_{2n} = Q_n^2$ and the system is only weakly hyperbolic.
\end{theorem}

\subsection{Preliminary results}
Before defining the HyQMOM closure, we will first need the following relations for the monic orthogonal polynomials $Q_n$:
\begin{lemma}\label{lemmaa2}
	For all $n=0,1,\dots$;
	\begin{equation}
		\frac{\langle X Q_{n+1}' Q_{n} \rangle}{\langle Q_{n}^2 \rangle} = \sum_{k=0}^{n}a_k,
		\qquad
		\frac{\langle X^2 Q_n' Q_{n} \rangle }{\langle Q_{n}^2 \rangle}= n a_n+\sum_{k=0}^{n-1}a_k.
	\end{equation}
\end{lemma}
\begin{proof}
	First, let us remark that $X Q_{n} = Q_{n+1} + a_{n}Q_{n}+ b_nQ_{n-1}$, so that
	$$
	\langle X Q_{n+1}' Q_{n} \rangle 
	= (n+1) a_n \langle Q_{n}^2 \rangle+b_n \langle Q_{n+1}' Q_{n-1} \rangle,
	$$
	since each $Q_k$ is monic and orthogonal to any polynomial of degree at most $k-1$.
	Moreover, since $Q_{n+1}' = Q_n + (X-a_n)Q_n' -b_nQ_{n-1}'$, one can deduce:
	$$
	\langle Q_{n+1}' Q_{n-1} \rangle = \langle X Q_{n}' Q_{n-1} \rangle - n a_n \langle Q_{n-1}^2 \rangle .
	$$
	Then, using $b_n=\frac{\langle Q_{n}^2 \rangle}{\langle Q_{n-1}^2 \rangle}$:
	$$
	\frac{\langle X Q_{n+1}' Q_{n} \rangle}{\langle Q_{n}^2 \rangle}  
	= (n+1) a_n +\frac{ \langle X Q_{n}' Q_{n-1} \rangle - n a_n \langle Q_{n-1}^2 \rangle  }{\langle Q_{n-1}^2 \rangle}
	= a_n +\frac{ \langle X Q_{n}' Q_{n-1} \rangle }{\langle Q_{n-1}^2 \rangle},
	$$
	which allows to prove the first equality.
	
	For the second one, we still use $X Q_{n} = Q_{n+1} + a_{n}Q_{n}+ b_nQ_{n-1}$ to find
	$$
	\langle X^2 Q_n' Q_{n} \rangle = n a_n\langle Q_n^2 \rangle + b_n\langle X Q_n' Q_{n-1} \rangle
	= \left(n a_n + \frac{ \langle X Q_{n}' Q_{n-1} \rangle }{\langle Q_{n-1}^2 \rangle} \right)\langle Q_n^2 \rangle.
	$$
	Thanks to the first equality, this concludes the proof.
\end{proof}

Finally, the following result is needed to relate a generalized recurrence relation involving $Q_n$ to the constraints in \eqref{3const}:
\begin{theorem}\label{theorema2}
	For all $n=1,2,\dots$; let the monic polynomial $P_{2n+1}$ be given by
	\begin{equation}	
		P_{2n+1} = Q_n\left[ (X-\alpha_n)Q_n -\beta_nQ_{n-1} \right]
	\end{equation}
	where $\alpha_n$ and $\beta_n$ are some real numbers.
	Then, the following statements are equivalent:
	\begin{enumerate}[label=(\roman*)]
		\item $\langle P_{2n+1} \rangle =0$, $\langle P_{2n+1}' \rangle =0$ and $ \langle X P_{2n+1}' \rangle=0$.
		\item $\displaystyle\alpha_n=a_n = \frac{1}{n}\sum_{k=0}^{n-1}a_k$ and $\displaystyle\beta_n=\frac{2n+1}{n}b_n$.
	\end{enumerate}
\end{theorem}
\begin{proof}
	It is easy to see that
	\begin{equation}\label{eq:rel1}
		\langle P_{2n+1}  \rangle = \langle X Q_n^2  \rangle - \alpha_n\langle Q_n^2  \rangle = (a_n-\alpha_n)\langle Q_n^2  \rangle.
	\end{equation}
	Moreover, since $ P_{2n+1}' = 2(X-\alpha_n)Q_n'Q_n +Q_n^2 - \beta_n ( Q_n'Q_{n-1}+Q_n Q_{n-1}' )$,
	\begin{equation}\label{eq:rel2}
		\langle P_{2n+1}'  \rangle = 2n\langle Q_n^2\rangle + \langle Q_n^2    \rangle - n\beta_n\langle Q_{n-1}^2  \rangle 
		=\left[ (2n+1)b_{n} - n\beta_n \right] \langle Q_{n-1}^2 \rangle .
	\end{equation}
	And finally
	$$
	\langle X P_{2n+1}'  \rangle = 2 \langle X^2 Q_n'Q_n  \rangle -2n\alpha_n\langle Q_{n}^2  \rangle +\langle X Q_{n}^2  \rangle
	- \beta_n\langle X Q_n'Q_{n-1}  \rangle.
	$$
	Using \eqref{Qcoeff} and \cref{lemmaa2}, this leads to 
	\begin{equation}\label{eq:rel3}
		\langle X P_{2n+1}'  \rangle = \left[ (2n+1)a_n + 2\sum_{k=0}^{n-1}a_k -2n\alpha_n - \frac{\beta_n}{b_n}\sum_{k=0}^{n-1}a_k \right]\langle Q_{n}^2  \rangle.
	\end{equation}
	Equations \eqref{eq:rel1}, \eqref{eq:rel2} and \eqref{eq:rel3} allow to conclude the proof.
\end{proof}

\subsection{HyQMOM closure for $S_{2n+1}$} \label{HyQMOM}

With HyQMOM, the moments up to $S_{2n}$ are known, and a closure for $S_{2n+1}$ is sought that makes the moment system globally hyperbolic. The choice is not unique (see \cref{PC2} for a discussion of the case with $n=2$), so we favor closures that are relatively simple to compute for arbitrary $n$, and for which global hyperbolicity can be demonstrated explicitly for $n\le 9$ and is postulated for larger values of $n$.

\begin{theorem}[HyQMOM closure for $S_{2n+1}$]\label{theorem2}
	Let $Q_n$ be the monic orthogonal polynomial defined by \eqref{Qn+1}, with $Q_{-1}=0$, $Q_0=1$, and $R_{n+1}$ be the monic polynomial defined by
	\begin{equation}\label{Rn+1new}
		R_{n+1}(X) = (X - \alpha_n) Q_n(X) - \beta_n Q_{n-1}(X) .
	\end{equation}
	For all $n=1,2,\dots,9$; the scaled characteristic polynomial in \cref{theorem1} can be written as 
	\begin{equation}\label{decomposenew}
		P_{2n+1} (X) =  Q_n (X) R_{n+1} (X)
	\end{equation}
	if and only if the closure on $S_{2n+1}$, defined through the coefficient $a_{n}$, and the coefficients $\alpha_n$ and $\beta_n$ in \eqref{Rn+1new} are related to the recurrence coefficients $a_k$ and $b_k$ by
	\begin{equation}\label{alpha_beta}
		a_n=\alpha_n = \frac{1}{n} \sum_{k=0}^{n-1} a_k ,  \quad \beta_n = \frac{2n+1}{n} b_n .
	\end{equation}	
\end{theorem}
\begin{proof}
First, if  the scaled characteristic polynomial $P_{2n+1}$ is given by \eqref{decomposenew}, then 
equations \eqref{alpha_beta} follow directly from \cref{theorema2} and the properties \eqref{3const} of the characteristic polynomial. 

Conversely, using the closure $a_n = \alpha_n$, we just need to prove the relation \eqref{decomposenew}.
For that, it is easier to use the vector $\bfY = (a_1, b_2, \dots, a_{n-1}, b_n)$, which is uniquely defined from $(S_3,\dots,S_{2n})$.
We can then use the following relation to compute the coefficients of the characteristic polynomial $P_{2n+1}$:
\begin{equation}
(c_3,\dots,c_{2n})= -\frac{D S_{2n+1}}{D \bfY} \left(\frac{D (S_3,\dots,S_{2n})}{D\bfY}\right)^{-1}.
\end{equation}
Then, such computations are done using Matlab symbolic to check  \eqref{decomposenew} for $n=2,3,\dots,9$.
Algorithm~\ref{al:verif} gives the details of these computations and the Matlab source code can be found in \cref{matlabcode}.
For $n=1$ the result is obvious from the closure (see \cref{ssec:n=1}).
\end{proof}

\begin{algorithm2e}[tbp] 
	\caption{Verification of Theorem~\ref{theorem2}}\label{al:verif}
	\KwData{$(a_k)_{k=1,\dots,n-1}$,$(b_k)_{k=2,\dots,n}$}
	\BlankLine
	\KwResult{$P_{2n+1} -  Q_n  R_{n+1} $}
	\BlankLine
	$(a_0,b_0,b_1) \leftarrow (0,1,1)$\;
	$\displaystyle a_n \leftarrow \frac{1}{n} \sum_{k=0}^{n-1} a_k$\tcp*[f]{Set the closure}\;
	$(S_0,S_1,S_2) \leftarrow (1,0,1)$\;
	Initialize each scalar $Z_{k,p}$ to zero for $k\in\{-1,\dots,n\},p\in\{0,\dots,2n+1\}$\;
	$Z_{0,0}\leftarrow 1$ \tcp*[f]{Reverse Chebyshev algorithm}\;
	$Z_{0,1}\leftarrow 0$\;
	\For{$k\leftarrow 1$ \KwTo $n$}{
		$Z_{k,0}\leftarrow  b_kZ_{k-1,0}$\;
		$Z_{k,1} \leftarrow Z_{k,0}\left(a_k+ \dfrac{Z_{k-1,1}}{Z_{k-1,0}} \right) $\;
	}
	\For{$p\leftarrow 1$ \KwTo $2n$}{
		\For{$k\leftarrow 0$ \KwTo $\floor*{n+1-\frac{p}{2}}$}{
			$Z_{k,p+1} \leftarrow Z_{k+1,p-1} + a_{k}Z_{k,p}+b_{k}Z_{k-1,p+1}$\;
		}
		$S_{p+1}\leftarrow Z_{0,p+1} $\;
	}
	$\bfY \leftarrow (a_1, b_2, a_2, \dots, a_{n-1}, b_n)$\tcp*[f]{Computation of the $c_k$}\;
	$\displaystyle (c_3,\dots,c_{2n})\leftarrow -\frac{D S_{2n+1}}{D \bfY} 
	\left(\frac{D (S_3,\dots,S_{2n})}{D\bfY}\right)^{-1}$\;
	$\displaystyle (c_0,c_1,c_2)\leftarrow \left(\frac{1}{2} \sum^{2n+1}_{m=3} (m-2) S_{m} c_{m},
	- \sum_{m=3}^{2n+1} m S_{m-1} c_{m}, - \frac{1}{2} \sum^{2n+1}_{m=3} m S_{m} c_{m} \right)$\;
	$\displaystyle P_{2n+1}\leftarrow X^{2n+1} + \sum_{k=0}^{2n} c_k X^k$\;
	$Q_{-1}\leftarrow 0$\tcp*[f]{Computation of the polynomials $Q_k$}\;
	$Q_{0}\leftarrow 1$\;
	\For{$k\leftarrow 0$ \KwTo $n-1$}{
		$Q_{k+1}\leftarrow (X-a_k)Q_{k} - b_kQ_{k-1}$\;
	}
	$R_{n+1}\leftarrow (X-a_n)Q_{n} - \dfrac{2n+1}{n}b_nQ_{n-1}$\;
	return $P_{2n+1} -  Q_n  R_{n+1}$\tcp*[f]{Final verification}\;
\end{algorithm2e}

\begin{remark}
	For realizable moments, $\beta_n \ge 0$ in \eqref{Rn+1new}.
\end{remark}

The main result concerning global hyperbolicity is as follows.
\begin{theorem}\label{realroots}
	When $\beta_n > 0$, the $n+1$ roots of $R_{n+1}$ in \eqref{Rn+1new} are real-valued and bound and separate the $n$ roots of $Q_n$. 
\end{theorem}
\begin{proof}
	When $\beta_n >0$ (and thus $b_n > 0$), the confluent form of the Christoffel--Darboux formula yields
	$$
	\sum_{k=0}^{n} \frac{Q_k^2(X)}{\langle Q_k^2 \rangle} =
	\frac{Q_n(X) R'_{n+1}(X) - R_{n+1}(X) Q_n'(X)}{\langle Q_n^2 \rangle}
	> 0 ,
	$$
	and thus
	\begin{equation}\label{cdconf}
	Q_n(X) R'_{n+1}(X) - R_{n+1}(X) Q_n'(X) > 0 .
	\end{equation} 
	The $n$ roots of $Q_n$ are real and distinct \cite{gautschi04} and denoted by $x_1<x_2<\dots<x_n$.
	For any two consecutive roots $x_k$ and $x_{k+1}$, \eqref{cdconf} implies that $$R_{n+1}(x_k) R_{n+1}(x_{k+1}) < 0.$$  
	Then the polynomial  $R_{n+1}$ has at least one root between $x_k$ and $x_{k+1}$, for each $k=1,\dots,n-1$.
	Moreover, $Q_n'(x_n)>0$, because of the behavior of $Q_n$ at $+\infty$. 
	Then $R_{n+1}(x_n)<0$ and $R_{n+1}$ has a root larger than $x_n$.
	In the same way, it has also a root smaller than $x_1$. This concludes the proof.
\end{proof}

\begin{remark}
	When $\beta_n=0$, $n$ roots of $R_{n+1}$ are shared with $Q_n$ and the root $\alpha_n$ has multiplicity of either 1 or 3. The latter occurs when $Q_{n}(\alpha_n)=0$, e.g., due to symmetry.
\end{remark}

In summary, the HyQMOM closure in \cref{theorem2} is globally hyperbolic for the moment system associated with the 1-D kinetic equation, and is well defined for any realizable moment vector $\bfM_{2n}$.

\subsection{Computation of $M_{2n+1}$ and the eigenvalues}\label{cs2np1}

Formulas for the closure $S_{2n+1}(S_3,\dots,S_{2n})$ could be given analytically.
This is done in the next section for $n=1$ and $n=2$, but these formulas becomes increasing more complicated due to the number of moments involved.
In the general case, the Chebyshev and reverse Chebyshev algorithms (see \cref{CBA}) can be used to compute $S_{2n+1}$.
But the transformation into standardized moment is not necessary, thanks to \cref{rem:linkab}, 
as well as \cref{th:dependence}.
Indeed, the characteristic polynomial $\bP_{2n+1}$ of the system of moments $\bfM_{2n}$ is linked to $P_{2n+1}$ by \eqref{eq:polychar}, the monic orthogonal polynomials $\bQ_k$ corresponding to $\bfM_{2n}$ are linked to the $Q_k$ by \eqref{eq:linkQ}. One can therefore define $\balpha_n = \bu +\sqrt{C_2} \alpha_n$, $\bbeta_n = \sqrt{C_2} \beta_n$ and $\bR_{n+1}(X) =C_2^{(n+1)/2} R_{n+1}\left(\frac{X-\bu}{\sqrt{C_2}}\right)$ in such a way that in \cref{theorem2}, the relations \eqref{alpha_beta} are equivalent to 
\begin{equation}\label{balpha_bbeta}
\balpha_n = \frac{1}{n} \sum_{k=0}^{n-1} \ba_k ,  \quad \bbeta_n = \frac{2n+1}{n} \bb_n,
\end{equation}
\eqref{Rn+1new} is equivalent to 
\begin{equation}
	\bR_{n+1}(X) = (X - \balpha_n) \bQ_n(X) - \bbeta_n \bQ_{n-1}(X) ,
\end{equation}
and \eqref{decomposenew}  is equivalent to 
\begin{equation}
	\bP_{2n+1} (X) =  \bQ_n (X) \bR_{n+1} (X).
\end{equation}
Then, $M_{2n+1}$ can be found from $\bfM_{2n}$ by Algorithm~\ref{al:closure}: the $\ba_k$ and $\bb_k$ are first computed from the moments by the Chebyshev algorithm, the value of $\ba_n$ is given by the closure $\ba_n=\balpha_n$, and the reverse Chebyshev algorithm then allows to compute $M_{2n+1}$.

The eigenvalues of the corresponding system, i.e., the roots of $\bQ_n$ and $\bR_{n+1}$, are then the eigenvalues of the following Jacobi matrices:
\begin{equation}\label{eq:JacobiQ}
\bfJ_n=
\begin{pmatrix}
	\ba_0         & \sqrt{\bb_1} &            &                &\\
	\sqrt{\bb_1}  & \ba_1        & \sqrt{\bb_2} &                &\\
                & \ddots     & \ddots     & \ddots	       &\\ 
                &            & \sqrt{\bb_{n-2}} & \ba_{n-2}        & \sqrt{\bb_{n-1}} \\ 
                &            &            & \sqrt{\bb_{n-1}} & \ba_{n-1}              
\end{pmatrix}
\end{equation}
and
\begin{equation}\label{eq:JacobiR}
\bfK_{n+1}=
\begin{pmatrix}
	\ba_0         & \sqrt{\bb_1} &            &                &\\
	\sqrt{\bb_1}  & \ba_1        & \sqrt{\bb_2} &                &\\
                & \ddots     & \ddots     & \ddots	       &\\ 
                &            & \sqrt{\bb_{n-1}} & \ba_{n-1}        & \sqrt{\bbeta_{n}} \\ 
                &            &            & \sqrt{\bbeta_n} & \balpha_n              
\end{pmatrix}.
\end{equation}

\begin{algorithm2e}[tbp] 
	\caption{Computation of $M_{2n+1}$}\label{al:closure}
	\KwData{$\bfM_{2n}$ strictly realizable}
	\BlankLine
	\KwResult{$M_{2n+1}$}
	\BlankLine
	Initialize each scalar $\sigma_{k,p}$ to zero for $k\in\{-1,\dots,n\},p\in\{0,\dots,2n+1\}$\;
	\For(\tcp*[f]{Chebyshev algorithm}){$p\leftarrow 0$ \KwTo $2n$}{
		$\sigma_{0,p} \leftarrow M_{p}$\;
	}
	$\ba_0 \leftarrow \frac{M_1}{M_0}$\;
	$\bb_0 \leftarrow 0$\;
	\For{$k\leftarrow 1$ \KwTo $n-1$}{
		\For{$p\leftarrow k$ \KwTo $2n-k$}{
			$\sigma_{k,p} \leftarrow \sigma_{k-1,p+1} - \ba_{k-1} \sigma_{k-1,p} -\bb_{k-1}  \sigma_{k-2,p}$\;
		}
		$\ba_{k} \leftarrow \dfrac{\sigma_{k,k+1}}{\sigma_{k,k}} - \dfrac{\sigma_{k-1,k}}{\sigma_{k-1,k-1}}$\;
		$\bb_{k} \leftarrow  \dfrac{\sigma_{k,k}}{\sigma_{k-1,k-1}}$ \;
	}
	$\sigma_{n,n} \leftarrow \sigma_{n-1,n+1} - \ba_{n-1} \sigma_{n-1,n} -\bb_{n-1}  \sigma_{n-2,n}$\;
	$\bb_n \leftarrow  \dfrac{\sigma_{n,n}}{\sigma_{n-1,n-1}}$\;
	$\displaystyle \ba_n \leftarrow \frac{1}{n}\sum_{k=0}^{n-1} \ba_k$\tcp*[f]{Set the closure}\;
	$\sigma_{n,n+1} \leftarrow \sigma_{n,n}\left(\ba_n+ \dfrac{\sigma_{n-1,n}}{\sigma_{n-1,n-1}} \right) $\tcp*[f]{Reverse Chebyshev algorithm}\;
	\For{$k\leftarrow n-1$ \KwTo $0$}{
		$\sigma_{k,2n-k+1} \leftarrow \sigma_{k+1,2n-k} + \ba_{k}\sigma_{k,2n-k}+\bb_{k}\sigma_{k-1,2n-k}$\;
	}
	$M_{2n+1} \leftarrow \sigma_{0,2n+1}$\;
	
\end{algorithm2e}

\section{Examples of the HyQMOM closure for $n \le 5$} \label{sec:example}

In this section, we apply the HyQMOM closure for $S_{2n+1}$ from \cref{theorem2} with \cref{theorem1} to find the characteristic polynomial $P_{2n+1}$ for $n \le 5$. Example plots are shown to illustrate the behavior of the roots of these polynomials as a function of $H_{2n}$ (i.e., distance from the boundary of moment space). For completeness, we begin with the trivial case $n=1$.

\subsection{$n=1$}\label{ssec:n=1}

As first shown in \cite{flv2018}, here $S_3 =0$ and the characteristic polynomial is
\begin{equation}\label{eq:P3}
P_{3} = X (X^{2} - 3) = Q_1 R_2  \quad \Longrightarrow \quad 
R_2 = X^{2} - 3 = XQ_1-\beta_1Q_0.
\end{equation}
Thus, for $n=1$, there are three real-valued roots ($0, \pm \sqrt{3}$), which correspond to the root of $Q_1$ and the two roots of $R_2$. As is well known in the literature, the moment system (i.e., the 1-D Euler equations) with ($M_0,M_1,M_2$) is globally hyperbolic.

\subsection{$n=2$}

\begin{figure}[tbp]
	\centering
	\includegraphics[width=0.99\textwidth]{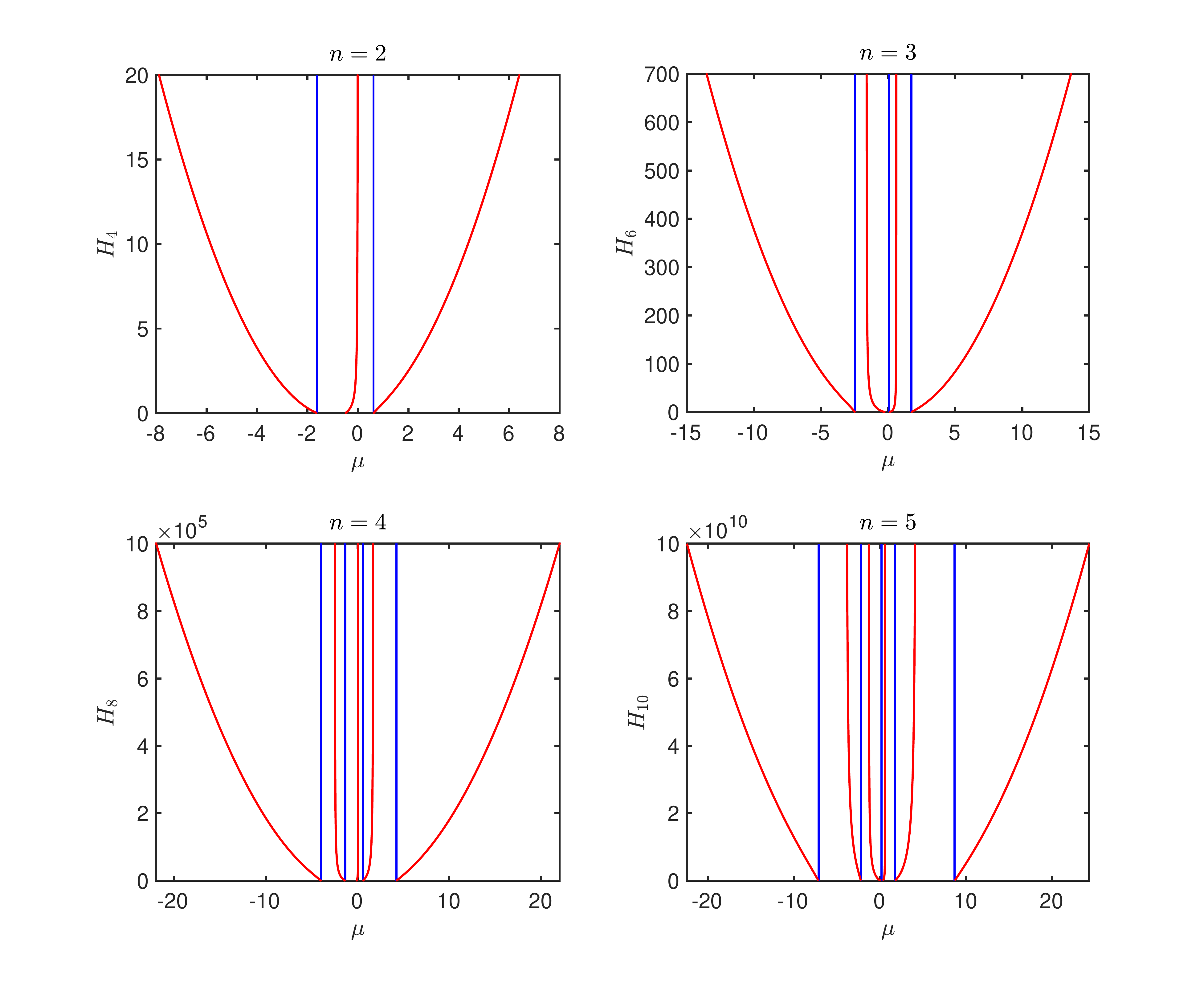}
	\caption{Roots $\mu$ of $Q_n$ (blue lines) and $R_{n+1}$ (red lines) and thus of $P_{2n+1}=Q_nR_{n+1}$ as functions of $H_{2n}$ for $n=2$ and $S_3=-1$ (top left), $n=3$ and $(S_3,S_4,S_5)=(-1,5,-8)$ (top right), $n=4$ and $(S_3,S_4,S_5,S_6,S_7)=(-1,5,-8,67.3,-100)$ (bottom left) and $n=5$ for $(S_3,S_4,S_5,S_6,S_7,S_8,S_9)=(-1,5,-8,67.3,-100,3000,0)$ (bottom right). For clarity, the lower-order standardized moments are held constant, but would likely also vary in real applications.}
	\label{fig:fig1}
\end{figure}

Here $\alpha_2=\frac{1}{2} S_3$ yields $S_5 = \frac{1}{2} S_3 (   5 S_4 - 3 S_3^2 - 1)$. The coefficients of the characteristic polynomial $P_5$ are
\begin{multline}\label{eq:PX+2}
		c_5 = 1, \
		c_4 = - \frac{5}{2} S_3 , \
		c_3 = \frac{1}{2}( -4 + 4 S_3^2 - 5 H_4 ) , \
		c_2 = \frac{1}{2} S_3 (6 - S_3^2 + 5 H_4 ) , \\
		c_1 = \frac{1}{2} ( 2 - 2 S_3^2 +5 H_4 ) , \
		c_0 = - \frac{1}{2} S_3 .
\end{multline}
When $H_4=0$, the middle root is located at $\alpha_2$.  Otherwise, there are two real-valued roots at $\frac{1}{2}[ S_3 \pm (4+S_3^2)^{1/2}]$ (i.e., the roots of $Q_2$), and three distinct real-valued roots from
\begin{equation}\label{eq:R2}
R_3  =
X^3 - \frac{3}{2} S_3 X^2 + \frac{1}{2}( -2 + S_3^2 - 5 H_4 ) X + \frac{1}{2} S_3.
\end{equation}
\cref{fig:fig1} illustrates the behavior of the roots for $S_3=-1$ with varying $H_4$. 
The roots of $Q_2$ are independent of $H_4$ and the two roots of $R_3$ join those of $Q_2$ when $H_4$ tends to zero.
Moreover, we can also see how the absolute value of the extremal roots of $R_3$ increase with $H_4$. 
We can also remark that with Gaussian moments, $Q_2$ is the Hermite polynomial $He_2$ and $R_3=X^3-6X$.

\subsection{$n>3$}\label{examplen>3}

For $n=3,4,5$, the characteristic polynomial $P_{2n+1}$ is found from $\alpha_n = a_n$.
In the same way as for $n=2$, 
\cref{fig:fig1} illustrates the behavior of its roots with fixed values of $(S_3,\dots,S_{2n-1})$ and varying $S_{2n}$, or equivalently $H_{2n}$.
For $n=3$, we took $(S_3,S_4,S_5)=(-1,5,-8)$; for $n=4$, $(S_3,S_4,S_5,S_6,S_7)=(-1,5,-8,67.3,-100)$; and for $n=5$, $$(S_3,S_4,S_5,S_6,S_7,S_8,S_9)=(-1,5,-8,67.3,-100,3000,0).$$
The same type of behavior as for the case $n=2$ is observed. Notwithstanding, for large $n$, the roots of $Q_n$ depend on all the standardized moments up to $S_{2n-1}$. Thus, the roots can be very different depending, for example, on how close the moment vector is to the boundary of moment space (which determines the lower bound on the even-order standardized moments).

Let us also remark that for $n=3$, with Gaussian moments, $Q_3$ is the Hermite polynomial $He_3$ and $R_4=X^4-10 X^2 + 9$.
In the case $n=4$, with nine roots depending on six parameters, the root locations can vary greatly for different values of the central moments, and will be very different from the roots of the Hermite polynomial $He_9$ used in Grad's moment closure \cite{grad1949}. 
For $n=5$, $Q_5$ is the Hermite polynomial $He_5$. For nearly Gaussian, asymmetric moment sets, the roots of $Q_n$ will be slightly displaced from those of the Hermite polynomial.

For the example in \cref{sec:tc}, calculations are done for $n$ up to 20 with no particular difficulties. In practice, the only foreseeable difficulty with using the HyQMOM closure for even larger $n$ is that associated with the accuracy of using the Chebyshev algorithm to find the recurrence coefficients from a realizable moment vector \cite{gautschi04}.

\section{Numerical example for the 1-D kinetic equation}\label{sec:tc}

In order to illustrate the predictions of the HyQMOM closure, we consider the 1-D Riemann problem from \cite{cflmv2017,flv2018}, for which the analytical solution for the VDF can be used to find reference solutions for the moments \cite{cflmv2017}. The initial condition for the mean velocity has a step at $x=0$:
$$
\bu = 
\begin{cases}
	+1 & \text{if $x< 0$,} \\
	-1 &  \text{if $x \ge 0$.} 
\end{cases}
$$
Otherwise, for all $x$, the initial moments correspond to a Maxwellian distribution function (i.e., $S_{2k+1}=0$ and $S_{2k+2}=(2k+1)S_{2k}$ for $k\ge 1$) with $M_0=1$ and $C_2 = \frac{1}{3}$. Due to the discontinuous mean velocity, for $t > 0$ and starting near $x=0$, the VDF quickly becomes far from Maxwellian.  The analytical solution is given in \cite{cflmv2017}, and reported in \cref{fig:fig5,fig:fig6} at $t=0.1$. 

The moment system \eqref{eq:closure5} is solved numerically on the 1-D computational domain $-0.5 < x < 0.5$ discretized into 4000 finite volumes using a first-order HLL scheme \cite{toro1999}. The CFL number is set to 0.5. 
The maximum/minimum eigenvalues over the domain needed to define the HLL fluxes are computed at each time step.
They are the maximum/minimum roots of $\bR_{n+1}$ and these roots are computed from the coefficients $\ba_k$ and $\bb_k$ by computing the eigenvalues of the corresponding Jacobi matrix given by \eqref{eq:JacobiR}. 
For comparison, the results found using the Gaussian, the Gaussian-EQMOM, and the entropy maximization closures are given in Fig.~2 of \cite{cflmv2017}. Results with $n=2$ for the previous definition of HyQMOM (i.e., with $a_2=0$) and QMOM are given in Figs.~1 and 2 of \cite{flv2018}.

\begin{figure}[tbp]
	\centering
	\includegraphics[width=0.99\textwidth]{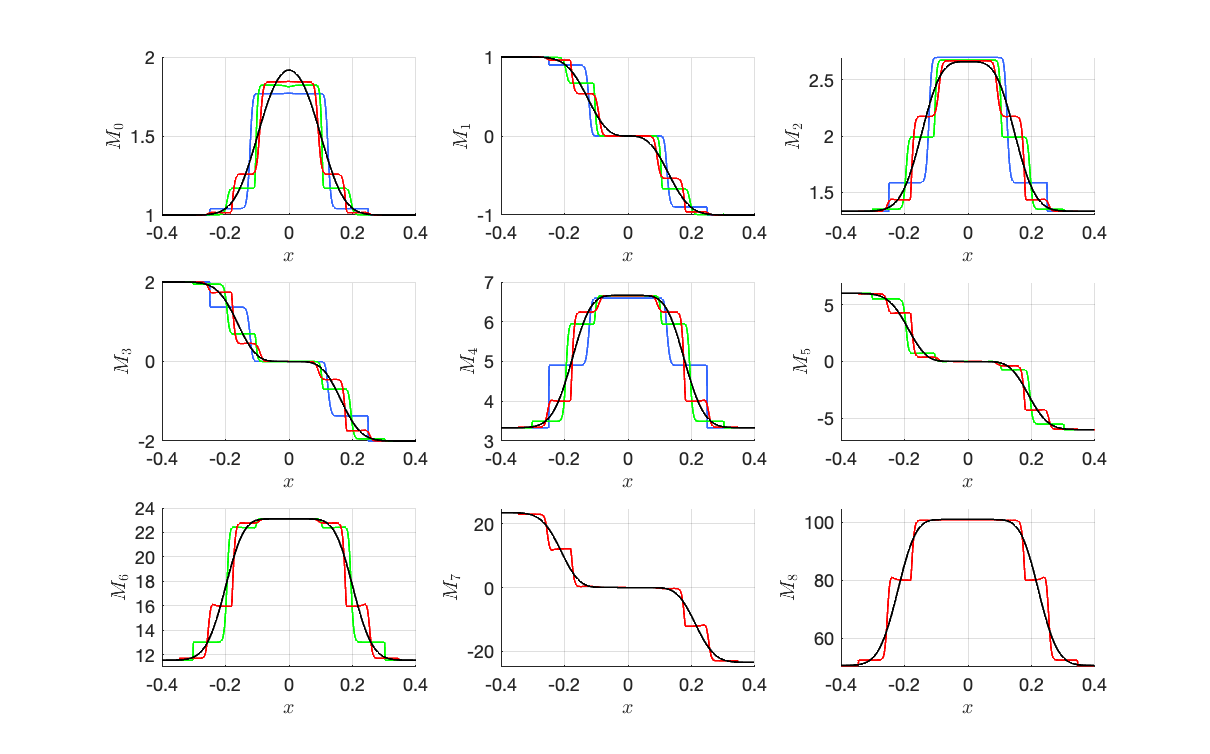}
		
	\caption{Numerical solution at $t=0.1$ of 1-D Riemann problem for the moments. Black, analytical solution. Blue, $n=2$. Green, $n=3$. Red, $n=4$.}
	\label{fig:fig5}
\end{figure}

\begin{figure}[tbp]
	\centering
	\includegraphics[width=0.99\textwidth]{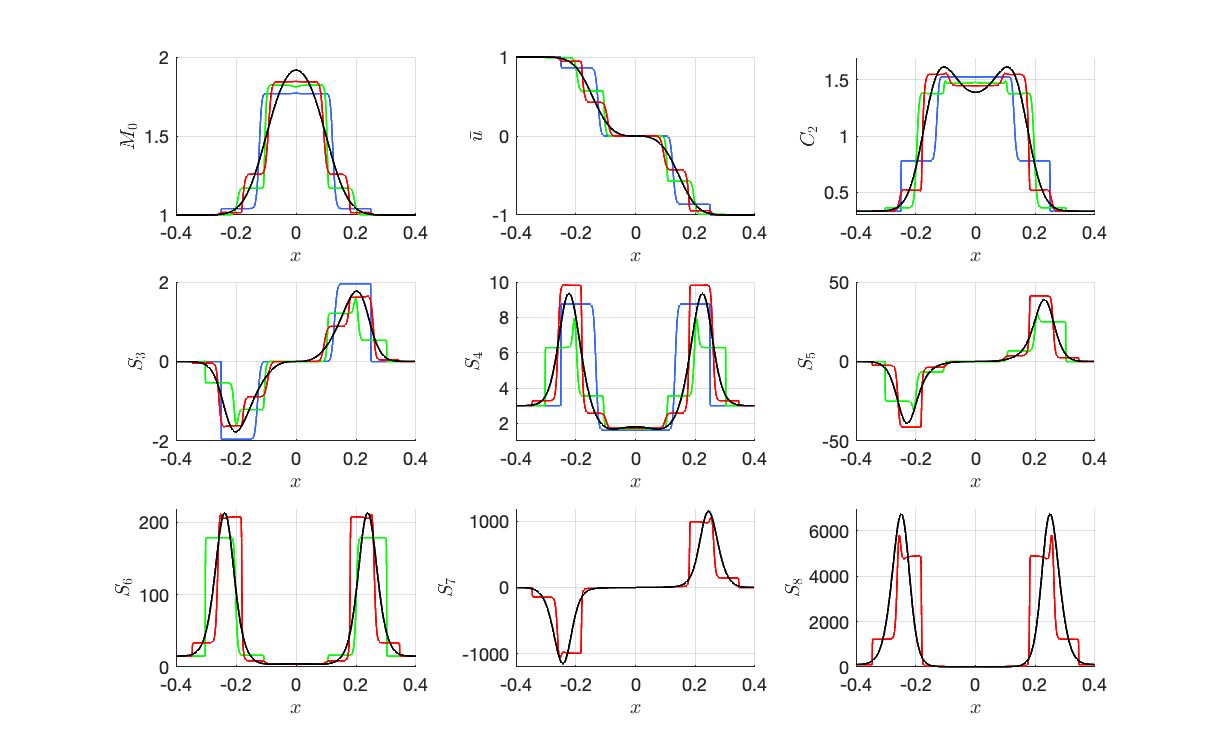}
	
	\caption{Numerical solution at $t=0.1$ of 1-D Riemann problem for the standardized moments. Black, analytical solution. Blue, $n=2$. Green, $n=3$. Red, $n=4$.}
	\label{fig:fig6}
\end{figure}

Qualitatively, the HyQMOM results for $n=2$ with $a_2 = \frac{1}{2} S_3$ are better than with $a_2=0$. This is likely due to the fact that the three eigenvalues are different due to their dependence on $a_2$. For example, when $S_3 < 0$ (i.e., $x < 0$), the most negative eigenvalue with $a_2 = \frac{1}{2} S_3$ has larger magnitude than with $a_2=0$. Thus, information propagates faster towards the left, giving, for example, a better approximation of the moments when compared to the analytical solution. This is clearly seen for $S_3$ where the location of the minimum (maximum) is better predicted with $a_2 = \frac{1}{2} S_3$, as is the maximum of $S_4$.  With $a_2=0$, the latter is significantly under-predicted (by a factor of two). In summary, for $n=2$ using the HyQMOM closure from this work yields more accurate predictions for the eigenvalues and thus for the moments. Notwithstanding, both definitions lead to globally hyperbolic moment systems, indicating that even within the family of hyperbolic closures improvements are possible by modifying the eigenvalues.

As can be seen from \cref{fig:fig5,fig:fig6}, the HyQMOM prediction improves with increasing $n$. As expected for a hyperbolic system with $2n+1$ degrees of freedom, the different speeds associated with the eigenvalues result in sub-shocks that are not present in the infinite-dimensional analytical solution. Remarkably, as $n$ increases, the speeds adapt to better capture the shapes of the moment profiles. We should remind the reader that these speeds are not known in advance (i.e., unlike with Hermite expansions), but adapt to the changing moments in a highly nonlinear manner. Although this non-linearity makes the analysis of the moment system challenging, it is a significant strength of QBMM because it allows the characteristic speeds to reflect very accurately the underlying moments. 
For example, with Gaussian moments a subset of the speeds correspond to the roots of a Hermite polynomial, while on the boundary of moment space they reduce to QMOM as required by the known form of the VDF. 

To examine convergence, the simulations were done without any trouble up to $n=20$, showing the robustness of the method.
Moreover, for the final solution, the maximal eigenvalue in absolute value is about 3.34 for $n=2$, 5.32 for $n=10$, and 6.5 for $n=20$; so they do not increase drastically.
The $L_2$ norm of the error for each moment is then computed for each simulation and normalized by the $L_2$ norm of the analytical solution.
These results are plotted on \cref{fig:fig7} as functions of $n$, varying from 2 to 20. 
Each line corresponds to a moment, from $M_0$ to $M_{20}$, with a gradation of the color from red to blue and then from blue to green.
The group of curves at the bottom corresponds to even-order moments, whereas the top group corresponds to odd-order moments, with (for fixed $n$) an error that increases with the order. Example results for the moments with $n=10$ are given in \cref{res10}. With $n=20$ the curves nearly overlap and cannot be distinguished with the scaling used for the plots.
Based on these results and \cref{fig:fig7}, the HyQMOM closure appears to converge with increasing $n$.
As a measure of the computation cost, the case with $n=20$ required 22~mins using Matlab on a laptop computer.

\begin{figure}[tbp]
	\centering
	\includegraphics[width=0.8\textwidth]{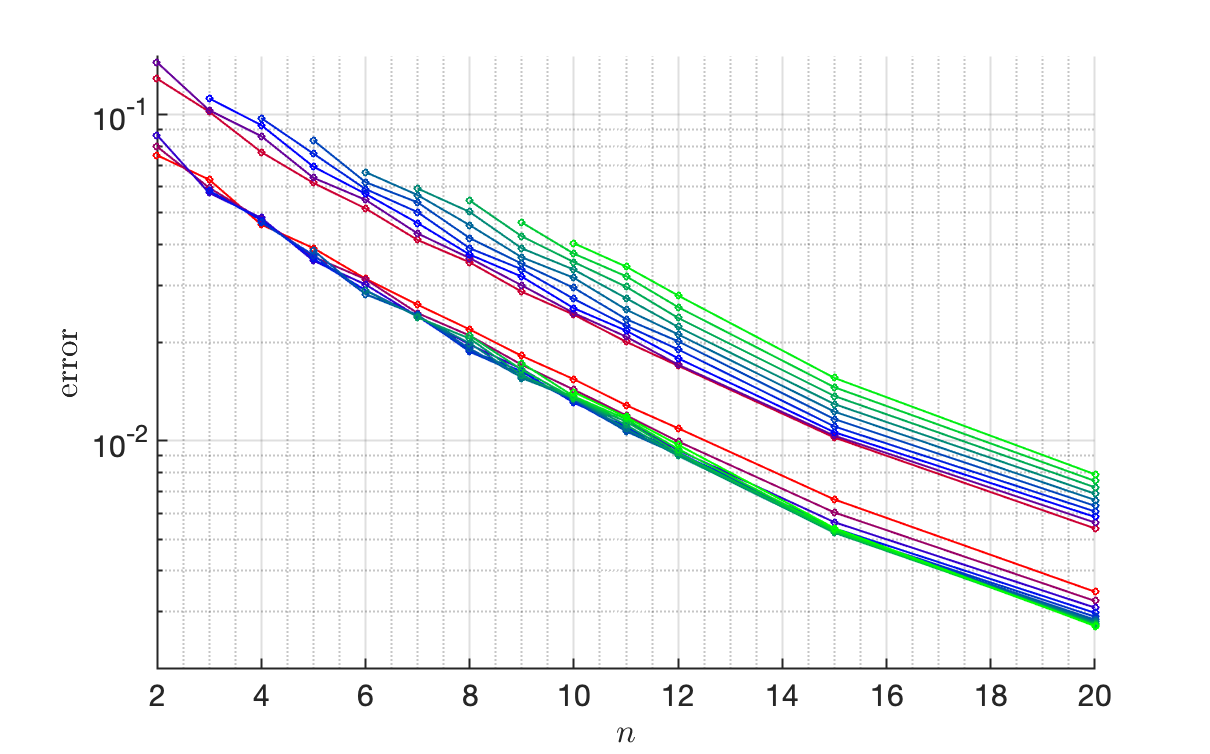}
	
	\caption{Error on the moments for the numerical solution at $t=0.1$ of the 1-D Riemann problem depending on $n$. Color gradation from red (for $M_0$) to blue (for $M_5$) and from blue (for $M_5$) to green ($M_{20}$).}
	\label{fig:fig7}
\end{figure}

\section{Discussion and conclusions}\label{sec:con}

Despite its apparently simple form, the path from the original HyQMOM closure for $n=2$ to the general HyQMOM closure in \cref{theorem2} was not straightforward. In \cite{flv2018}, the 1-D HyQMOM closure for $n=2$ was discovered by forcing 
an abscissa of a representing VDF to be located at the mean velocity (this induces $a_2=0$).
We therefore first sought to extend this condition to $2<n$, but were unsuccessful at finding a globally hyperbolic closure. By relaxing this condition, we were able to find closures with $\alpha_n \neq 0$ that are hyperbolic in restricted regions of moment space. However, the first real breakthrough came from the realization that as $H_{2n} \to 0$, we must have $P_{2n+1}(X) \to (X-\alpha_n) Q_n^2 (X)$. When combined with the condition $\langle P_{2n+1} \rangle = 0$, this implies that on the boundary of moment space we must have $a_n = \alpha_n$, which provides a closure for $S_{2n+1}$. Thus, using the results from  \cref{th:dependence} and \cref{theorem1}, we next sought functional forms for $a_n$ depending only on the standardized moments. 

This approach turned out to be relatively fruitful, enabling us to find closures up to $n=5$ that, at least numerically, appeared to be globally hyperbolic. Nevertheless, it was impossible to prove global hyperbolicity as the expressions were too complicated to advance analytically. However, we did observe that all such closures were \emph{nearly} of the form $P_{2n+1} = Q_n R_{n+1} + remainder$, where the remainder term went to zero as $H_{2n} \to 0$. As discussed in \cref{PC2} for $n=2$, the family of closures with $P_{2n+1} = Q_n R_{n+1}$ is not unique, nor is it usually possible to prove global hyperbolicity for larger $n$. But searching only for candidate polynomials $R_{n+1}$ greatly narrowed the field of possible closures. The decisive final breakthrough was \cref{theorema2}, which led us directly to  \cref{theorem2}. The classical proof of the relationship between the roots of $Q_n$ and $R_{n+1}$ in \cref{realroots} followed immediately, and established global hyperbolicity at least for $n\le 9$.

From a computational standpoint, Algorithm~\ref{al:closure} for computing $M_{2n+1}$ (and its extension in \cref{algo_closure}) is extremely efficient, especially when compared to the other candidate closures discussed above for which the cost of computing $M_{2n+1}$ becomes intractable for $5 \le n$. Likewise, the Jacobi matrices in \eqref{eq:JacobiQ} and \eqref{eq:JacobiR} provide the eigenvalues of the free-transport term at very little additional cost. As demonstrated in \cref{sec:tc}, this efficiency allowed us to test the convergence of the HyQMOM closure for $n=20$ (i.e., up to $M_{41}$). Nonetheless, it should be possible to go to even larger $n$ using Algorithm~\ref{al:closure} if needed. 

Because it was developed to control the eigenvalues of the free-transport term in the 1-D kinetic equation, one can ask whether the HyQMOM closure will be useful for closing other terms such as collisions or drag exchange with a second phase \cite{mf2013}. In any case, a VDF as a weighted sum of $n+1$ Dirac delta functions found from $\bfJ_{n+1}$ defined in \eqref{eq:JacobiQ} can be uniquely associated with the HyQMOM closure. The pragmatic response is that HyQMOM will be at least as good as QMOM with $n+1$ weights and abscissas  
for evaluating unclosed integrals with respect to the unknown velocity VDF. 

For population balance equations defined on semi-infinite or finite domains, it is not obvious that HyQMOM will provide any advantage relative to QMOM. Given that HyQMOM uses the even-order moment $M_{2n}$ while QMOM uses the odd-order moment $M_{2n+1}$, it will be necessary to prove that the HyQMOM closure for $M_{2n+1}$ is realizable for semi-infinite and finite domains. If this is the case, then the HyQMOM closure presented in this work can be used to investigate convergence with increasing $n$. Besides the kinetic equation and population balances, other potential applications of the HyQMOM closure include radiation transport \cite{vhwf2013,flz2020,flz2020jcp} and multiphase-flow models derived from a kinetic equation \cite{mf2013}.

Our current research is focused on the extension of 1-D HyQMOM to multiple dimensions (e.g., 2-D and 3-D) and infinite domains. In prior work with $n=2$ \cite{flv2018,pdf2019}, this was accomplished using the conditional QMOM \cite{yuan2011} and a reconstructed VDF based on the weights and abscissas corresponding to $Q_3$. For our extension of 1-D HyQMOM, we eschew that approach, and instead seek to close the  multi-dimensional moments appearing in the free-transport flux vector using ideas developed in the present work. Our initial results in this direction are promising from both an analytical and computational perspective.

\appendix

\section{Chebyshev and reverse Chebyshev algorithms}\label{CBA}

Let us consider a strictly realizable moment vector $\bfM_{N}$ and the associated linear functional $\langle . \rangle_{\bfM_{N}}$ on $\mathbb{R}[X]_{N}$ defined by \eqref{eq:defangle}.
Let us also define the sequence $(Q_k)_{k=0,\dots,n}$ of monic orthogonal polynomials for the scalar product $(p,q)\mapsto \langle pq \rangle_{\bfM_{N}}$ of $\mathbb{R}[X]_{n}$, with $n=\floor*{\frac{N}{2}}$.
The coefficients $(a_k)_{k=0,\dots,\floor*{\frac{N-1}{2}}}$ and $(b_k)_{k=0,\dots,\floor*{\frac{N}{2}}}$ of the recurrence relation for the monic orthogonal polynomials ($Q_{k+1}=(X-a_k)Q_k-b_kQ_{k-1}$ with $Q_0=1$, $Q_{-1}=0$) can be computed from the moments thanks to several algorithms: Rutishauser's QD algorithm \cite{Rutishauser54,Henrici58}, Gordon's PD algorithm \cite{Gordon68_1,Gordon68_2}, and a variation of an algorithm attributed to Chebyshev and given by Wheeler in \cite{w1974}.
However, except for the last one (referred to as the Chebyshev algorithm here), these algorithms first compute the variables $(\zeta_k)_{k=1,\dots,N}$ such that $b_k=\zeta_{2k-1}\zeta_{2k}$ and $a_k=\zeta_{2k}+\zeta_{2k+1}$.
They can then fail as a result of symmetries in a VDF corresponding to $\bfM_{N}$: the values $\zeta_k$ are indeed well defined and positive when the support of the VDF is included on $(0,+\infty)$, whereas some of them cannot be defined in some cases, like when the VDF corresponding to the moments is symmetric.
That is why we use only the Chebyshev algorithm.

As explained for example in \cite{w1974,gautschi04,Chebyshev1859}, the Chebyshev algorithm introduces the quantities $\sigma_{k,p}=\langle Q_kX^p \rangle_{\bfM_{N}}$ and uses the formulas
\begin{align}
	&b_k = \frac{\sigma_{k,k}}{\sigma_{k-1,k-1}}
	&\qquad&a_k = \frac{\sigma_{k,k+1}}{\sigma_{k,k}} - \frac{\sigma_{k-1,k}}{\sigma_{k-1,k-1}}
	\\
	&\sigma_{k+1,p}  =  \sigma_{k,p+1}  -  a_{k} \sigma_{k,p}  -  b_{k} \sigma_{k-1,p} 
	&\qquad&
	p \ge k+1 .
\end{align}
This is given in Algorithm~\ref{al:chebyshev}.
\begin{algorithm2e}[tbp] 
	\caption{Computation of the coefficients of the recurrence relation from the moments}\label{al:chebyshev}
	\KwData{$\bfM_{N}$ strictly realizable}
	\BlankLine
	\KwResult{$(a_k)_{k=0,\dots,\floor*{\frac{N-1}{2}}}$ and $(b_k)_{k=0,\dots,\floor*{\frac{N}{2}}}$}
	\BlankLine
	$n\leftarrow\floor*{\frac{N}{2}}$\;
	Initialize each scalar $\sigma_{k,p}$ to zero for $k\in\{-1,\dots,n\},p\in\{0,\dots,N\}$\;
	\For{$p\leftarrow 0$ \KwTo $N$}{
		$\sigma_{0,p} \leftarrow M_{p}$\;
	}
	$b_0 \leftarrow M_0$\;
	$a_0 \leftarrow \frac{M_1}{M_0}$\;
	\For{$k\leftarrow 1$ \KwTo $n$}{
		\For{$p\leftarrow k$ \KwTo $N-k$}{
			$\sigma_{k,p} \leftarrow \sigma_{k-1,p+1} - a_{k-1} \sigma_{k-1,p} -b_{k-1}  \sigma_{k-2,p}$\;
		}
		$b_{k} \leftarrow  \dfrac{\sigma_{k,k}}{\sigma_{k-1,k-1}}$ \;
		\If{$k<n$ or $N$ is odd}{
		$a_{k} \leftarrow \dfrac{\sigma_{k,k+1}}{\sigma_{k,k}} - \dfrac{\sigma_{k-1,k}}{\sigma_{k-1,k-1}}$\;
		}
	}		
\end{algorithm2e}

Conversely, from the coefficients $(a_k)_{k=0,\dots,n}$ and $(b_k)_{k=0,\dots,n}$, one can compute the moment vector $\bfM_{2n+1}$ through the reverse Chebyshev Algorithm~\ref{al:reversechebyshev}, using $Z_{k,p}=\sigma_{k,k+p}$.
Let us remark that it induces
\begin{equation}
Z_{k,0} = \prod_{j=0}^k b_j,\quad
Z_{k,1} = Z_{k,0}\sum_{j=0}^k a_j, \quad
Z_{k,p+1}=Z_{k+1,p-1} + a_{k}Z_{k,p}+b_{k}Z_{k-1,p+1},
\end{equation}
which imply that each $Z_{k,p}$, and then also $M_p$, is a multivariate polynomial function of the $a_j$ and $b_j$.

\begin{algorithm2e}[tbp] 
	\caption{Computation of the moments from the coefficients of the recurrence relation}\label{al:reversechebyshev}
	\KwData{$(a_k)_{k=0,\dots,n}$ and $(b_k)_{k=0,\dots,n}$}
	\BlankLine
	\KwResult{$\bfM_{2n+1}$}
	\BlankLine
	Initialize each scalar $Z_{k,p}$ to zero for $k\in\{-1,\dots,n\},p\in\{0,\dots,2n\}$\;
	$Z_{0,0}\leftarrow b_0$\;
	$Z_{0,1}\leftarrow b_0a_0$\;
	\For{$k\leftarrow 1$ \KwTo $n$}{
		$Z_{k,0}\leftarrow  b_kZ_{k-1,0}$\;
		$Z_{k,1} \leftarrow Z_{k,0}\left(a_k+ \dfrac{Z_{k-1,1}}{Z_{k-1,0}} \right) $\;
	}
	\For{$p\leftarrow 1$ \KwTo $2n$}{
		\For{$k\leftarrow 0$ \KwTo $\floor*{n-\frac{p}{2}}$}{
			$Z_{k,p+1} \leftarrow Z_{k+1,p-1} + a_{k}Z_{k,p}+b_{k}Z_{k-1,p+1}$\;
		}
	}
	\For{$p\leftarrow 0$ \KwTo $2n+1$}{
		$M_{p}\leftarrow Z_{0,p} $\;
	}
\end{algorithm2e}

\section{Hyperbolic closures for $n=2$}\label{PC2}

Let us consider the case with five moments. 
The closure $S_5(S_3,S_4)$ allows to compute the characteristic polynomial of the system:
\begin{equation}
	P_5(X) = c_0 + c_1 X + c_2 X^2 + c_3 X^3 + c_4 X^4 + X^5
\end{equation}
with
\begin{equation}
	\begin{gathered}
		c_4=- \frac{\partial S_5}{\partial S_4},
		\quad
		c_3=- \frac{\partial S_5}{\partial S_3},
		\quad
		c_2 = -\frac{1}{2}\left(3S_3c_3+4S_4c_4+S_5\right),
		\\
		c_1 = -\left(3c_3+4S_3c_4+5S_5\right),
		\quad
		c_0 = \frac{1}{2}\left(S_3c_3+2S_4c_4+3S_5\right).
	\end{gathered}
\end{equation}
Also recall that $Q_2(X) = X^2 - S_3 X -1$. The following theorem addresses the existence of hyperbolic closures for which $P_5$ is divisible by $Q_2$.

\begin{theorem}[Hyperbolic closure for $n=2$]\label{theoremapp}
	The polynomial $P_5$ is divisible by $Q_2$, i.e., $P_5=Q_2R_3$ with $R_3$ a real monic polynomial of degree 3, 
	if and only if there exists a real number $\gamma$ such that
	\begin{equation}\label{eq:cond_n2}
		S_5(S_3,S_4) =S_3 \left(2 + S_3^2 + \frac{5}{2} H_4 \right)+ \gamma H_4 \sqrt{4+S_3^2} \, .
	\end{equation}
	Moreover, $R_3$ has three real roots, which separate those of $Q_2$, if and only if $\gamma \in\left]-\frac{5}{2},\frac{5}{2}\right[$ .
\end{theorem}
\begin{proof}
	Identifying the coefficients, $P_5=Q_2R_3$ with $R_3=X^3 +aX^2 +bX + c$ is equivalent to
	$$
	\begin{pmatrix}
		1 & 0 & 0 & 0  & -1  \\
		-S_3 & 1 & 0 & -1  & 0  \\
		-1 & -S_3 & 1 & \frac{3}{2}S_3  & 2S_4  \\
		0 & 1 & S_3 & -3  & -4S_4  \\
		0 & 0 & 1 & \frac{1}{2}S_3  & S_4  
	\end{pmatrix}
	\begin{pmatrix}
		a  \\
		b \\
		c \\
		c_3 \\
		c_4
	\end{pmatrix}
	=
	\begin{pmatrix}
		S_3  \\
		1 \\
		-\frac{5}{2}S_5 \\
		5S_4 \\
		-\frac{3}{2}S_5
	\end{pmatrix} .
	$$
	This system gives $c_3$ and $c_4$, and, with the change of variable $S_4 = H_4 +S_3^2+1$,
	$$
	\left. \frac{\partial S_5}{\partial H_4} \right|_{S_3} = \frac{S_5-S_3^2-2S_3}{H_4},
	\quad
	\left. \frac{\partial S_5}{\partial S_3} \right|_{H_4} 
	=\frac{S_3S_5 + 2S_3^4+12S_3^2+10H_4 +8}{4+S_3^2}.
	$$
	Let us define $Y(S_3,S_4)$ such that $S_5(S_3,S_4) = S_3 \left( 2 + S_3^2 + \frac{5}{2} H_4 + YH_4 \right)$.
	Then 
	$$
	\left.\frac{\partial Y}{\partial H_4} \right|_{S_3} = 0,\quad
	\left. S_3(4+S_3^2)\frac{\partial Y}{\partial S_3} \right|_{H_4} = -4Y.
	$$
	This leads to $Y=\gamma \frac{\sqrt{4+S_3^2}}{S_3}$, so that we obtain \eqref{eq:cond_n2}.
	
	Moreover, let $r_{\pm}=\frac{S_3}{2}\pm \frac{\sqrt{4+S_3^2}}{2}$ denote the roots of $Q_2$.
	The values of $R_3$ at these points are
	$$
	R_3(r_+) = -\frac{H_4}{2}(2\gamma+5)r_+, 
	\quad
	R_3(r_-) = \frac{H_4}{2}(5-2\gamma)(-r_-).
	$$
	Since $r_+>0$ and $r_-<0$, it is easy to see that $R_3(r_+) <0$ and $R_3(r_-)>0$ if and only if $\gamma \in\left]-\frac{5}{2},\frac{5}{2}\right[$, thus concluding the proof.
\end{proof}

The five-moment system with the closure \eqref{eq:cond_n2} is globally hyperbolic for any $\gamma \in\left]-\frac{5}{2},\frac{5}{2}\right[$.
Let us remark that \eqref{eq:cond_n2} is equivalent to 
$$
a_2=\frac{S_3}{2} + \gamma \sqrt{4+S_3^2}\, \left(= \frac{(1+2\gamma)r_++(1-2\gamma)r_-}{2}\right).
$$
So, for $\gamma=0$, one have $a_2=(a_0+a_1)/2$, which is the HyQMOM closure from \cref{theorem2}. Furthermore, the globally hyperbolic closure introduced in \cite{flv2018} corresponds to $Y= -\frac{1}{2}$ (and $a_2=0$). However, since $Y$ is constant, \cref{theoremapp} does not apply and $P_5$ found with this closure is not divisible by $Q_2$.

\section{Matlab program for the proof of \cref{theorem2}}\label{matlabcode}
A Matlab symbolic code can be used to verify the result of \cref{theorem2}. For example, with $n=9$ the code is as follows:

\begin{verbatim}
clear all
close all

N = 9
Nmom = 2*N+1;

a = sym('a',[1,N+1],'real');
b = sym('b',[1,N+1],'real');
a(1) = 0;
b(1) = 1;
b(2) = 1;
%     closure: a(N+1)
a(N+1) = sum(a(1:N))/N;
% reverse Chebyshev algorithm
sig = sym('sig',[N+2,2*N+3],'real');
sig = sym(zeros(N+2,2*N+3));
S = sym('S',[1 2*N+1],'real');
S(1) = 0;
S(2) = 1;
sig(2,2) = 1;
sig(2,3) = a(1);
for k = 2:N+1
    sig(k+1,2) = sig(k,2)*b(k);
    sig(k+1,3) = sig(k+1,2)*(a(k)+sig(k,3)/sig(k,2));
end
for p = 3:2*N+2
   for k = 2:floor(N+3-p/2)
      sig(k,p+1) = sig(k+1,p-1)+a(k-1)*sig(k,p)+b(k-1)*sig(k-1,p+1);
   end
   S(p-1) = sig(2,p+1);
end

Y = [a(2) reshape([b(3:N+1);a(3:N+1)],[2*N-2 1])'];
J=simplify(jacobian(S(3:2*N),Y(1:2*N-2)));

Sc = S(2*N+1);

DSc = sym('DSc',[1 2*N],'real');
DScDab = jacobian(Sc,Y(1:2*N-2));
DSc(3:2*N) = (J'\ DScDab(1:2*N-2)')';
clear sig DScDab Y

% characteristic polynomial
P = sym(zeros(1,Nmom+1));
P(1) = 1;
P(2:Nmom-2) = -DSc(Nmom-1:-1:3);
P(Nmom-1)=sum((3:Nmom-1).*S(3:Nmom-1).*DSc(3:Nmom-1))/2-Nmom*Sc/2;
P(Nmom)=sum((3:Nmom-1).*S(2:Nmom-2).*DSc(3:Nmom-1))-Nmom*S(Nmom-1);
P(Nmom+1)=-sum((1:Nmom-3).*S(3:Nmom-1).*DSc(3:Nmom-1))/2 ...
     +(Nmom-2)*Sc/2;
clear J DSc S Sc

% calc Q
Q  = sym('Q',[N+1,N+1],'real');
Q  = sym(zeros(N+1,N+1));
% Q0
Q(1,1) = 1;
% Q1
Q(2,1) = 1;
% Qk
for k = 3:N+1
  Q(k,1:k) = Q(k-1,1:k)-a(k-1)*[0,Q(k-1,1:k-1)] ...
             -b(k-1)*[0,0,Q(k-2,1:k-2)];
end

syms x
PP = poly2sym(P,x);
QN = poly2sym(Q(N+1,1:N+1),x);
QNm = poly2sym(Q(N,1:N),x);
clear Q P
RR = simplify(expand(PP-QN*((x-a(N+1))*QN-(2*N+1)/N*b(N+1)*QNm)))
\end{verbatim}

\clearpage

\section{Results for 1-D Riemann problem with $n=10$}\label{res10}

\begin{figure}[tbh]
	\centering
	\includegraphics[width=0.99\textwidth]{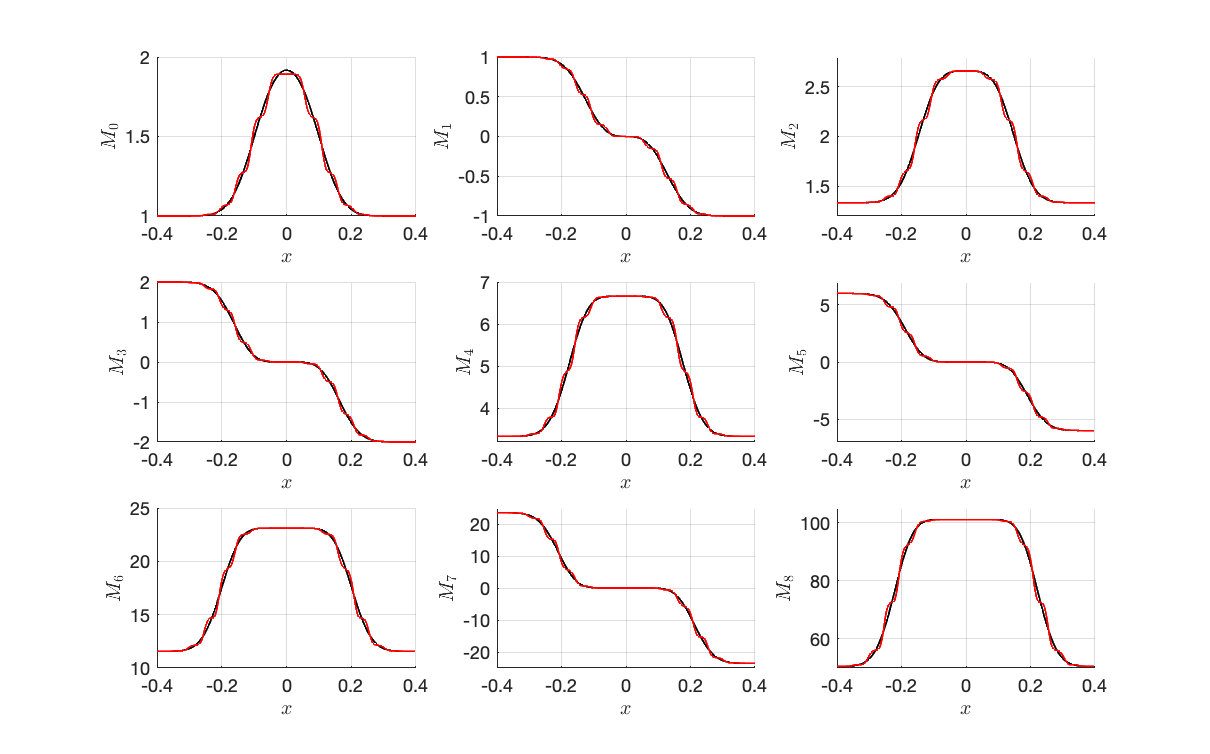}
	
	\includegraphics[width=0.99\textwidth]{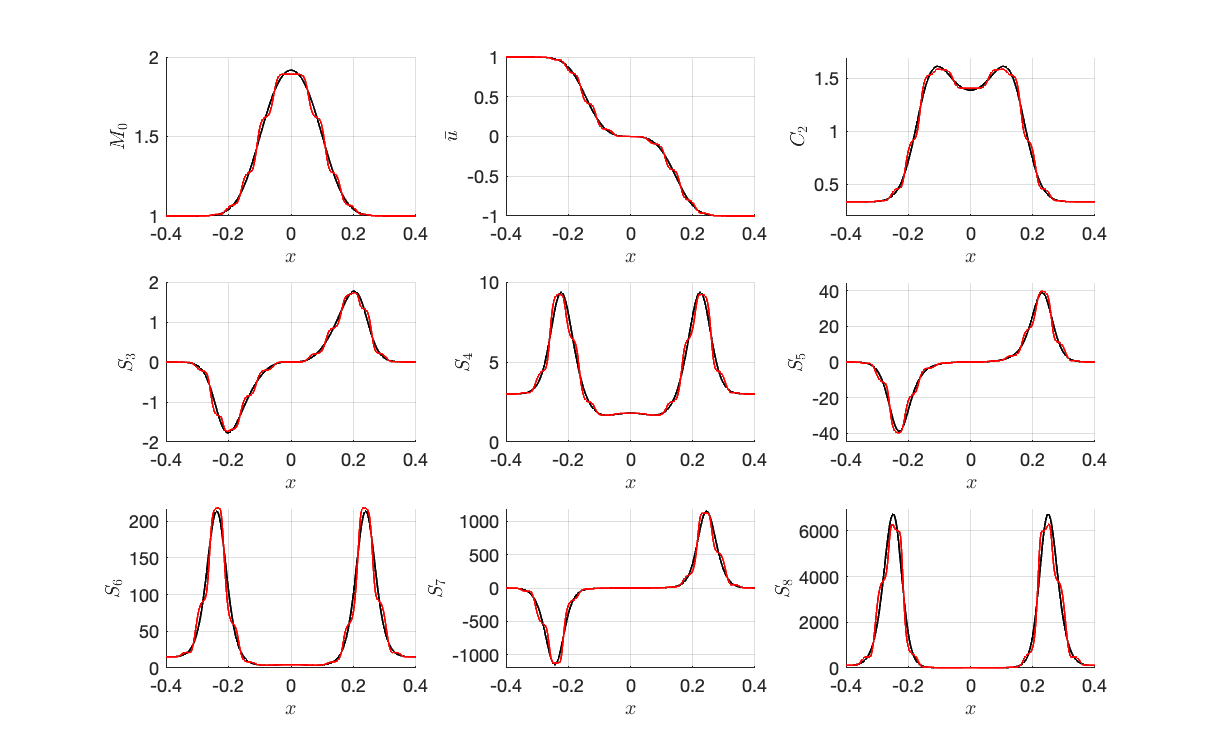}
	
	\caption{Numerical solution at $t=0.1$ of 1-D Riemann problem for the moments (top) and standardized moments (bottom). Black, analytical solution. Red, $n=10$.}
	\label{fig:fig10}
\end{figure}

\section{Algorithm for the computation of $M_{2n+1}$ for any realizable moment vector}\label{algo_closure}

Algorithm~\ref{al:closure} was restricted to a strictly realizable moment vector.
It can be generalized to any realizable moment vector, i.e., possibly on the boundary of moment space.
This includes the case $M_0=0$, where the VDF is zero and the case $C_2=0$ where the VDF is a Dirac delta function at $u=M_1/M_0$.
This also includes the case where $\bfM_{2n}$ is a sum of $k$ Dirac delta functions, with $k<n$, i.e., in such a case that $b_{k}=0$.
In this last case, the abscissas $x_p$ and weights $w_p$ for $p=1,\dots,k$ are computed from the Jacobi matrix $\bfJ_k$ defined by \eqref{eq:JacobiQ}. The abscissas are the eigenvalues of this matrix and the weights are computed from the eigenvectors \cite{gautschi04}.

\begin{algorithm2e}[tbh] 
	\caption{Computation of $M_{2n+1}$}\label{al:closure2}
	\KwData{$\bfM_{2n}$ realizable}
	\BlankLine
	\KwResult{$M_{2n+1}$}
	\BlankLine
	\If{$M_0=0$}{
		$M_{2n+1} \leftarrow 0$
	}
	\ElseIf{$M_2M_0-M_1^2=0$}{
		$M_{2n+1} \leftarrow \frac{M_1^{2n+1}}{M_0^{2n}}$\;
	}
	\Else{
	Initialize each scalar $\sigma_{k,p}$ to zero for $k\in\{-1,\dots,n\},p\in\{0,\dots,2n+1\}$\;
	\For{$p\leftarrow 0$ \KwTo $2n$}{
		$\sigma_{0,p} \leftarrow M_{p}$\;
	}
	$\ba_0 \leftarrow \frac{M_1}{M_0}$\;
	$\bb_0 \leftarrow M_0$\;
	$k \leftarrow 0$\;
	\While{$k\le n-2$ and $\bb_{k}>0$}{
		$k \leftarrow k+1$\;
		\For{$p\leftarrow k$ \KwTo $2n-k$}{
			$\sigma_{k,p} \leftarrow \sigma_{k-1,p+1} - \ba_{k-1} \sigma_{k-1,p} -\bb_{k-1}  \sigma_{k-2,p}$\;
		}
		$\ba_{k} \leftarrow \dfrac{\sigma_{k,k+1}}{\sigma_{k,k}} - \dfrac{\sigma_{k-1,k}}{\sigma_{k-1,k-1}}$\;
		$\bb_{k} \leftarrow  \dfrac{\sigma_{k,k}}{\sigma_{k-1,k-1}}$ \;
	}
	\If{$k=n-1$}{
	$\sigma_{n,n} \leftarrow \sigma_{n-1,n+1} - \ba_{n-1} \sigma_{n-1,n} -\bb_{n-1}  \sigma_{n-2,n}$\;
	$\bb_n \leftarrow  \dfrac{\sigma_{n,n}}{\sigma_{n-1,n-1}}$\;
	$\displaystyle \ba_n \leftarrow \frac{1}{n}\sum_{k=0}^{n-1} \ba_k$\;
	$\sigma_{n,n+1} \leftarrow \sigma_{n,n}\left(\ba_n+ \dfrac{\sigma_{n-1,n}}{\sigma_{n-1,n-1}} \right) $\;
	\For{$k\leftarrow n-1$ \KwTo $0$}{
		$\sigma_{k,2n-k+1} \leftarrow \sigma_{k+1,2n-k} + \ba_{k}\sigma_{k,2n-k}+\bb_{k}\sigma_{k-1,2n-k}$\;
	}
	$M_{2n+1} \leftarrow \sigma_{0,2n+1}$\;
	}
	\Else{
	compute the abscissas and weights $(x_p,w_p)_{p=1}^{k}$ from $(\ba_p,\bb_p)_{p=0}^{k-1}$\;
	$\displaystyle M_{2n+1} \leftarrow \sum_{p=1}^{k} w_p x_p^{2n+1}$\;
	}
	}
\end{algorithm2e}

\bibliographystyle{siamplain}
\bibliography{biblio_hyqmom}

\end{document}